\documentclass[11pt]{article}
\usepackage{amsmath, latexsym, amsfonts, amssymb, amsthm, amscd}
\usepackage[left=2.2cm, right=2.2cm, top=2.5cm, bottom=2.5cm]{geometry}
\usepackage{upquote}
\usepackage{color}
\usepackage{tipa}
\usepackage{setspace}
\usepackage[labelfont=bf]{caption}
\usepackage{stmaryrd}

\usepackage[utf8]{inputenc}
\usepackage[T1]{fontenc}
\usepackage[english]{babel}
\usepackage{dsfont}
\usepackage{mathtools} 
\usepackage{bm}
\usepackage{lmodern}
\usepackage{bbm}
\usepackage{mathrsfs}

\usepackage{xcolor}
\definecolor{Maroon}{HTML}{ad2231}
\definecolor{webgreen}{HTML}{008000}

\usepackage{hyperref}
\hypersetup{colorlinks, breaklinks, urlcolor=Maroon, linkcolor=Maroon, citecolor=webgreen} 

\allowdisplaybreaks

\usepackage{bookmark}

\usepackage{tikz}
\usetikzlibrary{calc}
\usepackage{pgfplots}



\newtheorem{corollary}{Corollary}
\newtheorem{proposition}{Proposition}
\newtheorem{lemma}{Lemma}

\newtheorem{theorem}{Theorem}

\theoremstyle{definition}
\newtheorem*{Assumption*}{Main assumption}

\newtheorem*{general*}{A general remark}

\begin{document}
\title{Invariance principle for fragmentation processes derived from conditioned stable Galton-Watson trees}
\author{Gabriel Berzunza Ojeda\footnote{ {\sc Department of Mathematical Sciences, University of Liverpool, United Kingdom.} E-mail: gabriel.berzunza-ojeda@liverpool.ac.uk }\, \, and \, \, Cecilia Holmgren\footnote{ {\sc Department of Mathematics, Uppsala University, Sweden.} E-mail: cecilia.holmgren@math.uu.se}}

\maketitle

\vspace{0.1in}

\begin{abstract} 
Aldous, Evans and Pitman (1998) studied the behavior of the fragmentation process derived from deleting the edges of a uniform random tree on $n$ labelled vertices. In particular, they showed that, after proper rescaling, the above fragmentation process converges as $n \rightarrow \infty$ to the fragmentation process of the Brownian CRT obtained by cutting-down the Brownian CRT along its skeleton in a Poisson manner.

In this work, we continue the above investigation and study the fragmentation process obtained by deleting randomly chosen edges from a critical Galton-Watson tree $\mathbf{t}_{n}$ conditioned on having $n$ vertices, whose offspring distribution belongs to the domain of attraction of a stable law of index $\alpha \in (1,2]$. Our main results establish that, after rescaling, the fragmentation process of $\mathbf{t}_{n}$ converges as $n \rightarrow \infty$ to the fragmentation process obtained by cutting-down proportional to the length on the skeleton of an $\alpha$-stable L\'evy tree of index $\alpha \in (1,2]$. We further show that the latter can be constructed by considering the partitions of the unit interval induced by the normalized $\alpha$-stable L\'evy excursion with a deterministic drift studied by Miermont (2001). This extends the result of Bertoin (2000) on the fragmentation process of the Brownian CRT. 

\texttt{
The proof of Theorem \ref{Theo3} has been corrected to address a gap in the tightness argument. Specifically, the issue arose from Lemma \ref{lemma5}, which is used in the proof of Theorem \ref{Theo3}. While the proof of convergence of the finite-dimensional distributions in Lemma \ref{lemma5} is correct, the gap lay in the proof of tightness. 
\\
We are grateful to Professor Svante Janson for pointing out this gap, which originated from the application of the unfortunately incorrect \cite[Lemma 22]{Brou2016}. Moreover, Professor Janson also helped us to fill this gap, and the complete proof of the tightness required for Theorem \ref{Theo3} can be found in \cite{GbSvCe2025}.
}
\end{abstract}


\noindent {\sc Subject Classes}: 60J25, 60J90, 60F05, 60G52, 60C05.

\section{Introduction and main results} \label{Sec1}

Aldous, Evans and Pitman \cite{AldousPitman1998, EvansPitman1998, Pitman1999} (see also \cite{Brou2016, MarckertWang2019}) considered a fragmentation process of a uniform random tree $\mathbf{t}_{n}$ on $n \in \mathbb{N}$ labelled vertices (or Cayley tree with $n$ vertices) by deleting the edges of $\mathbf{t}_{n}$ one by one in uniform random order. More precisely, as time passes, the deletion of edges creates more and more subtrees of $\mathbf{t}_{n}$ (connected components) such that the evolution of the ranked vector of sizes (number of vertices) of these subtrees (in decreasing order) evolves as a fragmentation process. It turns out that the asymptotic behavior of this fragmentation process, in reverse time, is related to the so-called {\sl standard additive coalescent} \cite{AldousPitman1998, EvansPitman1998}. Moreover, this leads to a continuous representation of the standard additive coalescent in terms of the time-reversal of an analogue fragmentation process of the Brownian continuum random tree (Brownian CRT); see \cite{AldousPitman1998}. Evans and Pitman \cite[Theorem 2]{EvansPitman1998} showed that an additive coalescent is a Feller Markov process with values in the infinite ordered set
\begin{eqnarray} \label{eq21}
\mathbb{S} \coloneqq \Big \{ \mathbf{x} = (x_{1}, x_{2}, \dots): x_{1} \geq x_{2} \geq \cdots \geq 0 \hspace*{2mm} \text{and} \hspace*{2mm} \sum_{i=1}^{\infty} x_{i} < \infty \Big \},
\end{eqnarray}

\noindent endowed with the $\ell^{1}$-norm, $\Vert \mathbf{x} \Vert_{1} = \sum_{i=1}^{\infty} |x_{i}|$ for $\mathbf{x} \in \mathbb{S}$, whose evolution is described formally by: given that the current state is $\mathbf{x}$, two terms $x_{i}$ and $x_{j}$, $i<j$, of $\mathbf{x}$ are chosen and merged into a single term $x_{i} + x_{j}$ (which implies some reordering of the resulting sequence) at rate equal to $x_{i} + x_{j}$. A version of this process defined for times describing the whole real axis is called {\sl eternal}. This model is also closely related to the so-called Marcus-Lushnikov process \cite{Marcus1968, lushnikov1978}, and in particular, the version studied in \cite{AldousPitman1998} is referred to as the standard additive coalescent.

In this work, we shall extend the investigation, that was begun in \cite{AldousPitman1998, EvansPitman1998, Pitman1999}, to the more general situation where one wants to cut-down critical Galton--Watson trees conditioned on having a fixed number of vertices, but whose offspring distribution belongs to the domain of attraction of a stable law. More precisely, consider a critical offspring distribution $\mu = (\mu(k), k \geq 0)$, i.e., a probability distribution on the nonnegative integers satisfying $\sum_{k = 0}^{\infty} k \mu(k) = 1$. In addition, we always implicitly assume that $\mu(0) + \mu(1) < 1$ (or equivalently, $\mu(0) >0$) to avoid degenerate cases, and that $\mu$ is aperiodic. We say that $\mu$ belong to the domain of attraction of a stable law of index $\alpha \in (1,2]$ if either the variance of $\mu$ is finite (in which case $\alpha =2$), or if $\mu([k, \infty)) = k^{-\alpha} L(k)$ as $k \rightarrow \infty$, where $L: \mathbb{R}_{+} \rightarrow \mathbb{R}_{+}$ is a function such that $L(x) >0$ for $x \in \mathbb{R}_{+}$ large enough and $\lim_{x \rightarrow \infty} L(tx)/L(x) = 1$ for all $t >0$ (such a function is called a slowly varying function). In other terms, if $(Y_{i})_{i \geq 1}$ is a sequence of i.i.d.\ random variables with distribution $\mu$, then there exists a sequence of positive real numbers $(B_{n})_{n \geq 1}$ such that 
\begin{eqnarray} \label{eq10}
B_{n} \rightarrow \infty \hspace*{3mm} \text{and} \hspace*{3mm} \frac{Y_{1} + Y_{2} + \cdots + Y_{n} - n}{B_{n}}  \xrightarrow[ ]{d} Z_{\alpha}, \hspace*{3mm} \text{in distribution as} \hspace*{2mm} n \rightarrow \infty,
\end{eqnarray}

\noindent where the Laplace transform of $Z_{\alpha}$ is given by $\mathbb{E}[\exp(-\lambda Z_{\alpha})] = \exp(\lambda^{\alpha})$ whenever $\alpha \in (1,2)$, and $\mathbb{E}[\exp(-\lambda Z_{2})] = \exp(\lambda^{2}/2)$ if $\alpha=2$, for every  $\lambda > 0$ (\cite[Section XVII.5]{Feller} guarantees its existence). In particular, for $\alpha =2$, we have that $Z_{2}$ is distributed as a standard Gaussian random variable. The factor $B_{n}$ is of order $n^{1/\alpha}$ (more precisely, $B_{n}/n^{1/\alpha}$ is a slowly varying sequence), and one may take $B_{n} = \sigma n^{1/2}$ when $\mu$ has finite variance $\sigma^{2}$. 

We henceforth let $\mathbf{t}_{n}$ denote a critical Galton--Watson tree conditioned on having $n$ vertices, whose offspring distribution $\mu$ belongs to the domain of attraction of a stable law of index $\alpha \in (1,2]$. We will refer to $\mathbf{t}_{n}$ as an $\alpha$-stable ${\rm GW}$-tree, for simplicity. Following Aldous, Evans and Pitman \cite{AldousPitman1998, EvansPitman1998}, we are interested in the evolution of the ranked vector of sizes (in decreasing order) of the subtrees created by deleting randomly chosen edges from $\mathbf{t}_{n}$. Indeed, we will consider a continuous-time version of this cutting-down process. Pick an $\alpha$-stable ${\rm GW}$-tree, say $\mathbf{t}_{n}$, and let $\textbf{edge}(\mathbf{t}_{n})$ be its set of edges and equip each of them with i.i.d.\ uniform random weights $\mathbf{w} = (w_{e}: e \in \textbf{edge}(\mathbf{t}_{n}))$ on $[0,1]$. For $u \in [0,1]$, we then keep the edges of $\mathbf{t}_{n}$ with weight smaller than $u$ and discard the others. Therefore, one obtains a (fragmentation) forest $\mathbf{f}_{n}(u)$ conformed by the connected components (or subtrees of $\mathbf{t}_{n}$) created by the above procedure; see Figure \ref{Fig3}. In particular, the forest $\mathbf{f}_{n}(u)$ has the same set of vertices as $\mathbf{t}_{n}$ but clearly it has a different set of edges given by $\textbf{edge}(\mathbf{f}_{n}(u)) = \{ e \in \textbf{edge}(\mathbf{t}_{n}): w_{e} \leq u\}$. Let $\mathbf{F}_{n} = (\mathbf{F}_{n}(u), u \in [0,1])$ be the process given by 
\begin{eqnarray*}
\mathbf{F}_{n}(u) = (F_{n,1}(1-u), F_{n,2}(1-u), \dots), \hspace*{3mm} \text{for} \hspace*{2mm} u \in [0,1], 
\end{eqnarray*}

\noindent the sequence of sizes (number of vertices) of the connected components of the forest $\mathbf{f}_{n}(1-u)$, ranked in decreasing order. We have strategically viewed the sequence of sizes of the components of $\mathbf{f}_{n}(1-u)$ as an infinite sequence, by completing with an infinite number of zero terms. 
Plainly as time passes more and more subtrees are created, and thus, the process $\mathbf{F}_{n}$ evolves as a fragmentation process. Note also that $\mathbf{F}_{n}(0) = (n, 0, 0, \dots)$ and that $\mathbf{F}_{n}(1) = (1, 1, \dots, 1, 0, 0, \dots)$, where the first $n$ terms in $\mathbf{F}_{n}(1)$ are ones and the remaining terms
are zeros. Since we are interested in studying the asymptotic behaviour of $\mathbf{F}_{n}$, we consider the (rescaled in time and space) fragmentation process $\mathbf{F}_{n}^{(\alpha)} = (\mathbf{F}_{n}^{(\alpha)}(t), t \geq 0)$ given by 
\begin{eqnarray} \label{eq15}
\mathbf{F}_{n}^{(\alpha)}(t) = \frac{1}{n} \mathbf{F}_{n}\left( \frac{B_{n}}{n} t \right), \hspace*{2mm} \text{for} \hspace*{2mm} 0 \leq t \leq n/B_{n}, \hspace*{2mm} \text{and} \hspace*{2mm}  \mathbf{F}_{n}^{(\alpha)}(t) = \frac{1}{n} \mathbf{F}_{n}(1) \hspace*{2mm} \text{for} \hspace*{2mm} t > n/B_{n},
\end{eqnarray}

\noindent where $(B_{n})_{n \geq 1}$ is a sequence satisfying (\ref{eq10}). The process $\mathbf{F}_{n}^{(\alpha)}$ takes values on the set $\mathbb{S}$. The aim of this paper is to establish a convergence result for the fragmentation process $\mathbf{F}_{n}^{(\alpha)}$. To state the precise statement (Theorem \ref{Theo3}), it will be convenient to first introduce the limiting object. 

Bertoin \cite{Bertoin2000} showed that the fragmentation process of the Brownian CRT in \cite{AldousPitman1998} can be constructed by considering the partitions of the unit interval induced by a standard Brownian excursion with drift. This latter is sometimes called the {\sl Brownian fragmentation}. In a similar vein, Miermont \cite{Miermont2001} built other fragmentation processes from  L\'evy processes with no positive jumps (or equivalently, negatives of spectrally positive L\'evy processes). Specifically, let $X_{\alpha}^{\rm exc} = (X_{\alpha}^{\rm exc}(s), s \in [0,1]$) be the normalized excursion (with unit length) of an $\alpha$-stable spectrally positive L\'evy process of index $\alpha \in (1,2]$; see Section \ref{Sec2}. In particular, $X_{2}^{\rm exc}$ is the normalized standard Brownian excursion. For every $t \geq 0$, define the processes $Y_{\alpha}^{(t)} = (Y_{\alpha}^{(t)}(s), s \in [0,1])$ and $I_{\alpha}^{(t)} = (I_{\alpha}^{(t)}(s), s \in [0,1])$ by letting
\begin{eqnarray} \label{eq14}
Y_{\alpha}^{(t)}(s) = X_{\alpha}^{\rm exc}(s) - ts \hspace*{2mm} \text{and} \hspace*{2mm} I_{\alpha}^{(t)}(s) = \inf_{u \in [0,s]} Y_{\alpha}^{(t)}(u), \hspace*{4mm} \text{for} \hspace*{2mm} s \in [0,1].
\end{eqnarray} 

\noindent For $t \geq 0$, we introduce 
\begin{eqnarray}
\mathbf{F}^{(\alpha)}(t) = (F^{(\alpha)}_{1}(t), F^{(\alpha)}_{2}(t), \dots)
\end{eqnarray}

\noindent as the random element of $\mathbb{S}$ defined by the ranked sequence (in decreasing order) of the lengths of the intervals components of the complement of the support of the Stieltjes measure ${\rm d} (- I_{\alpha}^{(t)})$; note that $s \mapsto -I_{\alpha}^{(t)}(s) = \sup_{u \in [0,s]} - Y_{\alpha}^{(t)}(u)$ is an increasing process. More precisely, the support of ${\rm d} (- I_{\alpha}^{(t)})$ is defined as the set of times when the process $Y_{\alpha}^{(t)}$ reaches a new infimum. On the other hand, it can be shown that the support of ${\rm d} (- I_{\alpha}^{(t)})$ coincides with the so-called ladder time set of $-Y_{\alpha}^{(t)}$ which is given by the closure of the set of times when $Y_{\alpha}^{(t)}$ is equal to its infimum, i.e.,
\begin{eqnarray*}
\mathscr{L}^{\alpha}(t) \coloneqq \overline{\left\{ s \in [0,1]:  Y_{\alpha}^{(t)}(s) = I_{\alpha}^{(t)}(s) \right\}}; 
\end{eqnarray*}

\noindent see for example \cite[Proposition 1, Chapter VI]{Bertoin1996} and the discussion after that. Then $\mathbf{F}^{(\alpha)}(t)$ is the lengths of the open intervals in the canonical decomposition of $[0,1]\setminus \mathscr{L}^{\alpha}(t)$ arranged in the decreasing order. The intervals components of the complement of the support of the measure ${\rm d} (- I_{\alpha}^{(t)})$ are also called constancy intervals of $-I_{\alpha}^{(t)}$, and in fact, those intervals corresponds to excursion intervals of $Y_{\alpha}^{(t)}$ above its infimum (or equivalently, excursion intervals of the reflected process $Y_{\alpha}^{(t)} - I_{\alpha}^{(t)}$ above $0$). It is well-known that $\mathscr{L}^{\alpha}(t)$ is a.s.\ a random closed set with zero Lebesgue measure which implies that $\mathbf{F}^{(\alpha)}(t) \in \mathbb{S}_{1}$ a.s., where $\mathbb{S}_{1} \subset \mathbb{S}$ is the space of the elements of $\mathbb{S}$ with sum $1$; see \cite[Corollary 5, Chapter VII]{Bertoin1996}. Observe that for every fixed $0 \leq t < t^{\prime}$, the process $s \mapsto Y_{\alpha}^{(t)}(s) - Y_{\alpha}^{(t^{\prime})}(s) = (t^{\prime} -t) s$ is monotone increasing which entails that $\mathscr{L}^{\alpha}(t) \subseteq \mathscr{L}^{\alpha}(t^{\prime})$. Then the partition of $[0,1]$ induced by $\mathscr{L}^{\alpha}(t^{\prime})$ is finer than that induced by $\mathscr{L}^{\alpha}(t)$. As a consequence, it has been shown by Miermont \cite[Proposition 2]{Miermont2001} (see also \cite[Theorem 1]{Bertoin2000} for $\alpha=2$) that $\mathbf{F}^{(\alpha)} = (\mathbf{F}^{(\alpha)}(t), t \geq 0)$ is a fragmentation process issued from $\mathbf{F}^{(\alpha)}(0) = (1,0,0, \dots)$. A precise description of its transition kernel is given in \cite[Definition 4]{Miermont2001}; see Corollary \ref{corollary1} below for some insights. From now on, we will refer to $\mathbf{F}^{(\alpha)}$ as the $\alpha$-stable fragmentation of index $\alpha \in (1,2]$. 

We are now able to state our first main result. Let $\mathbb{D}(I, \mathbb{M})$ be the space of c\`adl\`ag functions from an interval $I \subseteq \mathbb{R}$ to the separable, complete metric space $(\mathbb{M}, d)$ equipped with the Skorohod topology; (see e.g.\  \cite[Chapter 3]{Billi1999} or \cite[Chapter VI]{jacod2003} for details on this space). We write $\xrightarrow[ ]{d}$ to denote convergence in distribution.
\begin{theorem} \label{Theo3}
Let $\mathbf{t}_{n}$ be an $\alpha$-stable ${\rm GW}$-tree of index $\alpha \in (1,2]$. Then, we have that
\begin{eqnarray*}
(\mathbf{F}_{n}^{(\alpha)}(t), t \geq 0) \xrightarrow[ ]{d} (\mathbf{F}^{(\alpha)}(t), t \geq 0), \hspace*{3mm} \text{as} \hspace*{2mm}  n \rightarrow \infty, \hspace*{2mm} \text{in the space} \hspace*{2mm} \mathbb{D}(\mathbb{R}_{+}, \mathbb{S}).
\end{eqnarray*}
\end{theorem}

As mentioned earlier, $\mathbf{F}^{(2)}$ is exactly the Brownian fragmentation studied by Bertoin \cite{Bertoin2000}, that is to say, it corresponds to the fragmentation process derived from the Brownian CRT of Aldous and Pitman \cite{AldousPitman1998}; see also \cite{Abraham2002f}. In view of this, the second goal of this paper is to show that indeed $\mathbf{F}^{(\alpha)}$ is the fragmentation process obtained by cutting-down the ``edges'' of the $\alpha$-stable L\'evy tree.

It is well-known that the $\alpha$-stable L\'evy tree of index $\alpha \in (1,2]$ appears as scaling limits of $\alpha$-stable ${\rm GW}$-trees; see Duquesne and Le Gall \cite{DuLegall2002}. The $\alpha$-stable L\'evy tree $\mathcal{T}_{\alpha} = (\mathcal{T}_{\alpha}, d_{\alpha}, \rho_{\alpha})$ is a random compact metric space $(\mathcal{T}_{\alpha}, d_{\alpha})$ with one distinguished element $\rho \in \mathcal{T}_{\alpha}$ called the root such that $(\mathcal{T}_{\alpha}, d_{\alpha})$ is a tree-like space in that for $v,w \in \mathcal{T}_{\alpha}$, there is a unique non-self-crossing path $[v, w]$ from $v$ to $w$ in $\mathcal{T}_{\alpha}$, whose length equals $d_{\alpha}(v,w)$. The leaves ${\rm Lf}(\mathcal{T}_{\alpha})$ of $\mathcal{T}_{\alpha}$ are those points that do not belong to the interior of any path leading from one point to another, and the skeleton of the tree is the set ${\rm Sk}(\mathcal{T}_{\alpha}) = \mathcal{T}_{\alpha} \setminus {\rm Lf}(\mathcal{T}_{\alpha})$ of non-leaf points. The $\alpha$-stable L\'evy tree $\mathcal{T}_{\alpha}$ is naturally endowed with a uniform probability measure $\mu_{\alpha}$ ({\sl the mass measure}) that is supported on ${\rm Lf}(\mathcal{T}_{\alpha})$, and a unique $\sigma$-finite measure $\lambda_{\alpha}$ ({\sl the length measure}) carried by ${\rm Sk}(\mathcal{T}_{\alpha})$ that assigns measure $d_{\alpha}(v,w)$ to the geodesic path between $v$ and $w$ in $\mathcal{T}_{\alpha}$. 

Following Aldous-Pitman's fragmentation \cite{AldousPitman1998} of the Brownian CRT, the analogue of deleting randomly chosen edges in $\mathbf{t}_{n}$ is to cut the skeleton of $\mathcal{T}_{\alpha}$ by a Poisson point process of cuts with intensity ${\rm d}t \otimes \lambda_{\alpha}({\rm d}v)$ on $[0, \infty) \times \mathcal{T}_{\alpha}$. For all $t \geq 0$, define an equivalence relation $\sim_{t}$ on $\mathcal{T}_{\alpha}$ by saying that $v \sim_{t} w$, for $v, w \in \mathcal{T}_{\alpha}$, if and only if, no atom of the Poisson process that has appeared before time $t$ belongs to the path $[v, w]$. These cuts split the $\alpha$-stable L\'evy tree into a (continuum) forest, that is a countably infinite set of smaller subtrees (connected components) of $\mathcal{T}_{\alpha}$. Let $\mathcal{T}_{\alpha,1}^{(t)}, \mathcal{T}_{\alpha,2}^{(t)}, \dots$ be the distinct equivalence classes for $\sim_{t}$ (connected components of $\mathcal{T}_{\alpha}$), ranked according to the decreasing order of their $\mu_{\alpha}$-masses. The subtrees $(\mathcal{T}_{\alpha,i}^{(t)}, i \geq 1)$ are nested as $t$ varies, that is, for every $0 \leq t < t^{\prime}$ and $i \geq 1$, there exits $j \geq 1$ such that $\mathcal{T}_{\alpha,i}^{(t^{\prime})} \subset \mathcal{T}_{\alpha,j}^{(t)}$. Let $\mathbf{F}_{\mathcal{T}_{\alpha}} = (\mathbf{F}_{\mathcal{T}_{\alpha}}(t), t \geq 0)$ be the process given by
\begin{eqnarray*}
\mathbf{F}_{\mathcal{T}_{\alpha}}(t) = (\mu_{\alpha}(\mathcal{T}_{\alpha,1}^{(t)}),  \mu_{\alpha}(\mathcal{T}_{\alpha,2}^{(t)}), \dots), \hspace*{4mm} t \geq 0,
\end{eqnarray*}

\noindent where $\mathbf{F}_{\mathcal{T}_{\alpha}}(0) = (1, 0, 0, \dots)$. Indeed, $\mathbf{F}_{\mathcal{T}_{\alpha}}$ is a fragmentation process in the sense that $\mathbf{F}_{\mathcal{T}_{\alpha}}(t^{\prime})$ is obtained by splitting at random the elements of $\mathbf{F}_{\mathcal{T}_{\alpha}}(t)$, for $0 \leq t < t^{\prime}$. We call $\mathbf{F}_{\mathcal{T}_{\alpha}}$ the fragmentation process of the $\alpha$-stable L\'evy tree. In particular, $\mathbf{F}_{\mathcal{T}_{2}}$ is the fragmentation process of the Brownian CRT introduced in \cite[Section 2.2]{AldousPitman1998}. Note that $\mathbf{F}_{\mathcal{T}_{\alpha}}$ takes values in $\mathbb{S}$, and that Lemma \ref{lemma6} below shows that $\mathbf{F}_{\mathcal{T}_{\alpha}}(t) \in \mathbb{S}_{1}$ a.s., for every $t \geq 0$. We can now state our second main result.
\begin{proposition} \label{Theo4}
We have that 
\begin{eqnarray*}
(\mathbf{F}^{(\alpha)}(t), t \geq 0) \stackrel{d}{=} (\mathbf{F}_{\mathcal{T}_{\alpha}}(t), t \geq 0),
\end{eqnarray*}
\noindent where $\stackrel{d}{=}$ means equality of finite-dimensional distributions. 
\end{proposition}

Theorem 3 in \cite{AldousPitman1998} shows that the time-reversed fragmentation process of the Brownian CRT, i.e.\ $(\mathbf{F}_{\mathcal{T}_{2}}(e^{-t}), t \in \mathbb{R})$, is a version of the standard additive coalescent providing an explicit construction of this last process. In general, Miermont \cite[Section 6]{Miermont2001} has shown that the time-reversed $\alpha$-stable fragmentation  process, i.e.\ $(\mathbf{F}^{(\alpha)}(e^{-t}), t \in \mathbb{R})$, is an eternal additive coalescent as described by Evans and Pitman \cite{EvansPitman1998}. More precisely, it is a mixing of so-called {\sl extremal coalescents} of Aldous and Pitman \cite{AldousPitmanI2000} (see also \cite{Bertoin2001}) which exact law is given in \cite[Proposition 3]{Miermont2001}. Thus, Proposition \ref{Theo4} implies that this eternal additive coalescent can also be constructed from the $\alpha$-stable L\'evy tree by Poisson splitting along its skeleton. On the other hand, Theorem \ref{Theo3} and Proposition \ref{Theo4} clearly generalize Bertoin's work \cite{Bertoin2000} and moreover, complete Miermont's \cite{Miermont2001} one by identifying the distribution of the $\alpha$-stable fragmentation with that of the fragmentation process of the $\alpha$-stable L\'evy tree. In particular, Bertoin \cite{BertoinS2002} proved that $\mathbf{F}^{(2)}$ (or equivalently, $\mathbf{F}_{\mathcal{T}_{2}}$) is a so-called {\sl self-simlar fragmentation process} of index $1/2$. However, Miermont \cite{MiermontII} has already pointed out that $\mathbf{F}^{(\alpha)}$ (and therefore $\mathbf{F}_{\mathcal{T}_{\alpha}}$), for $\alpha \in (1,2)$, is not a self-similar fragmentation due to the existence of points in $\mathcal{T}_{\alpha}$ with infinite degree. \\

The proof of Theorem \ref{Theo3} uses some of the ideas developed in \cite{Brou2016} where only the case of the Cayley tree was treated. However, in our more general framework, there are technical challenges that do not appear in \cite{Brou2016}, mostly due to the lack of some properties that only the Cayley tree satisfies. To prove Theorem \ref{Theo3}, we use the so-called {\sl Prim's algorithm} \cite{Prim} to obtain a consistent ordering on the vertices of the forest created by deleting randomly chosen edges from $\mathbf{t}_{n}$ that we refer to as the {\sl Prim order}. Informally, given $\mathbf{t}_{n}$ whose edges are equipped with non-negative and distinct weights, and a starting vertex, say $v$ of $\mathbf{t}_{n}$, Prim's algorithm explores a connected component from $v$, each time visiting a neighbouring vertex whose connecting edge possesses the smallest weight; see Section \ref{Sec3}. Then every time an edge is removed and a new connected component is created, the Prim order of the vertices in the new forest always remains the same. This will allow us to precisely encode this forest (and in particular, the sizes of connected components) using a discrete analogue of the process $Y^{(t)}_{\alpha}$ defined in (\ref{eq14}) that we refer to as the {\sl Prim path}. We then show that this (properly rescaled) Prim path indeed converges to its continuous version. Finally, we develop a general approach for the convergence of fragmentation processes encoded by functions in $\mathbb{D}([0,1], \mathbb{R})$ to conclude our proof. 

There are some of the key differences with the proof for Cayley trees in \cite{Brou2016}. For example, the convergence of the encoding processes in \cite{Brou2016} uses a bound (see in (10) in \cite{Brou2016}) that is only known to hold for Cayley trees (or Galton--Watson trees where $\mu$ has some exponential finite moment). In \cite{Brou2016}, the authors mostly work with convergence of continuous processes. This is no longer possible in our framework, since our encoding processes are discontinuous due to the nature of the $\alpha$-stable ${\rm GW}$-trees. The above makes an important difference at the technical level. 

The proof of Proposition \ref{Theo4} follows along the lines of that of Theorem 3 in \cite{AldousPitman1998} for the Brownian CRT (see also the proof of Proposition 13 in \cite{AldousPitmanI2000}). Informally, we use the convergence of rescaled $\alpha$-stable ${\rm GW}$-trees toward the $\alpha$-stable L\'evy tree $\mathcal{T}_{\alpha}$ in order to approximate the fragmentation process of  $\mathcal{T}_{\alpha}$.  

The rest of the paper is organized as follows. In Section \ref{SecModels}, we discuss some connections with some combinatorial and probabilistic models: additive coalescents, parking schemes, laminations and Bernoulli bond-percolation. In Section \ref{Sec2}, we recall some facts about stable L\'evy processes, bridges and excursions that will be important for our proofs. Section \ref{Sec3} is devoted to the introduction of Galton--Watson trees as well as the formal definition of the exploration process (the Prim path) associated with the fragmentation forest. The asymptotic behavior of the Prim path is studied in Section \ref{Sec4}. Finally, the proofs of Theorem \ref{Theo3} and Proposition \ref{Theo4} are given in Section \ref{Sec7} and Section \ref{Sec6}, respectively.

\section{Further remarks} \label{SecModels}
In this section, we comment on our main results and highlight some connections with previous works. \\

\noindent \textbf{Additive coalescents.} The Cayley tree of size $n$ can be viewed as a Galton--Watson tree with Poissonian offspring distribution of parameter $1$ and conditioned to have $n$ vertices, where the labels are assigned to the vertices uniformly at random. In particular, the fragmentation process studied in \cite{AldousPitman1998, EvansPitman1998, Pitman1999}, say $\mathbf{F}_{n}^{+} = (\mathbf{F}_{n}^{+}(t), t \geq 0)$, corresponds to $\mathbf{F}_{n}^{(\alpha)}$ in (\ref{eq15}), with $\alpha =2$ and $B_{n} = n^{1/2}$. The fragmentation process $\mathbf{F}_{n}^{+}$ leads to a representation of an additive coalescent by an appropriate time reversal, that is, the exponential time-change $t \rightarrow e^{-t}$. Specifically, $(\mathbf{F}_{n}^{+}(e^{-t}), t \geq -(1/2) \ln n)$ is an additive coalescent starting at time $-(1/2) \ln n$ from the state $(1/n, 1/n, \dots, 1/n, 0,0, \dots) \in \mathbb{S}$ (or equivalently, from the component sizes in Marcus-Lushnikov model with $n$ initial masses $1/n$). Evans and Pitman \cite{EvansPitman1998} (see also \cite[Proposition 2]{AldousPitman1998}) showed that this time-reversed version of $\mathbf{F}_{n}^{+}$ converges in distribution to the standard additive coalescent, i.e., $(\mathbf{F}_{\mathcal{T}_{2}}(e^{-t}), t \in \mathbb{R})$. 

Aldous and Pitman \cite{AldousPitmanI2000} (see also \cite[Construction 5]{EvansPitman1998}) also studied the fragmentation process derived by cutting-down {\sl birthday trees}. They are a family of trees that generalizes the Cayley tree in allowing ``weights'' on the vertices. Aldous and Pitman showed that this fragmentation process, suitable rescaled, converges to the fragmentation process associated of the continuum counterpart of birthday trees, the {\sl inhomogeneous continuum random trees} (ICRT). Moreover, the time-reversed version of the fragmentation process of the ICRT can be viewed as version of an eternal additive coalescent. On the other hand, Bertoin \cite{Bertoin2001} has proved that the fragmentation process of the ICRT can also be constructed by considering the partitions of the unit interval induced by certain bridges with exchangeable increments.\\

\noindent \textbf{Parking schemes.} Chassaing and Louchard \cite{Chasing2002} have provided yet another representation of the standard additive coalescent as parking schemes related to {\sl Knuth’s parking problem}; see also \cite{ChassainSv2001, MarckertWang2019}. Bertoin and Miermont \cite{BertoinMiermont2006C} extended the work \cite{Chasing2002} and relate Knuth's parking problem for caravans to different versions of eternal additive coalescent. On the other hand, Knuth's parking problem bears some similarities with the dynamics of an aggregating server studied by Bertoin \cite{Bertoin2001} that also relate to the additive coalescent. \\

\noindent \textbf{Lamination process.} In a recent work, Th\'evenin \cite{thvenin2019geometric} has provided a geometric representation of the fragmentation process $\mathbf{F}_{\mathcal{T}_{\alpha}}$ by a new lamination-valued process. In particular, Theorem 1.1 in \cite{thvenin2019geometric} combined with Proposition \ref{Theo3} allows to deduce the exact distribution of the ranked sequence (in decreasing order) of the masses of the faces of this lamination-valued process. \\

\noindent \textbf{Bernoulli bond-percolation.} Bernoulli bond-percolation on finite connected graphs is perhaps the simplest example of a percolation model. In this model, each edge in the connected graph is removed with probability $1-p \in (0,1)$, and it is kept with probability $p$, independently of the other edges. This induces a partition of the set of vertices of the graph into connected components usually  referred to as clusters. It should be intuitively clear that there is a link between Bernoulli bond-percolation on $\alpha$-stable ${\rm GW}$-trees and their associated fragmentation processes. More precisely, let $\mathbf{t}_{n}$ be an $\alpha$-stable ${\rm GW}$-tree. For $u \in [0,1]$, recall that continuous-time cutting-down procedure of $\mathbf{t}_{n}$ described in the introduction results in a random forest of connected components. Indeed, the probability that a given edge of $\mathbf{t}_{n}$ has not yet been removed at time $u$ is exactly $u$. Thus, the configuration of the connected components at time $u$ is precisely that resulting from Bernoulli bond-percolation on $\mathbf{t}_{n}$ with parameter $u$. A natural problem in this setting is then to investigate the asymptotic behavior of the sizes (number of vertices) of the largest clusters for appropriate percolation regimes. In this direction, let $(B_{n})_{n \geq 1}$ be a sequence of positive real numbers satisfying (\ref{eq10}). An application of Theorem \ref{Theo3} shows that for the percolation parameter $1-(B_{n}/n)t$ with a fixed $t \geq 0$, the sequence of sizes of the clusters ranked in decreasing order and renormalized by a factor of $1/n$ (i.e.\ $\mathbf{F}_{n}^{(\alpha)}(t)$) converges in distribution, as $n \rightarrow \infty$, to $\mathbf{F}^{(\alpha)}(t)$. In particular, Theorem 2 in \cite{Miermont2001} allows us to describe explicitly the distribution of $\mathbf{F}^{(\alpha)}$ at fixed times. Let $(p_{s}(z), z \in \mathbb{R}, s \geq 0)$ be the family of densities of the distribution of a strictly stable spectrally positive L\'evy process with index $\alpha \in (1, 2]$; see Section \ref{Sec2}. 
\begin{corollary} \label{corollary1}
For $t >0$, let ${\rm a}_{1}^{(\alpha)}(t) > {\rm a}_{2}^{(\alpha)}(t)  > \cdots$ be the atoms of a Poisson measure on $(0, \infty)$ with intensity $\Lambda_{\alpha}^{(t)}({\rm d} z) \coloneqq  z^{-1}p_{z}(-tz)  \mathds{1}_{\{ z >0 \}} {\rm d} z$, ranked in decreasing order. Then
\begin{eqnarray*}
\mathbf{F}^{(\alpha)}(t) \stackrel{d}{=}  \Big( ({\rm a}_{1}^{(\alpha)}(t), {\rm a}_{2}^{(\alpha)}(t), \dots) \, \bigg | \sum_{i=1}^{\infty} {\rm a}_{i}^{(\alpha)}(t) = 1 \Big).
\end{eqnarray*} 
\end{corollary}

Following Bertoin's \cite{BertoinPer2013} work about Bernoulli bond-percolation on random trees. The percolation regime $1-(B_{n}/n)t$ on $\mathbf{t}_{n}$ corresponds to the so-called supercritical regime. Indeed, the result in Corollary \ref{corollary1} has already been proved by Pitman \cite{Pitman1999} for Cayley trees. Furthermore, it has been shown in \cite{AldousPitman1998, Bertoin2000} that the distribution of $\mathbf{F}^{(2)}(t)$ is equal to that of the ranked jump sizes (in decreasing order) of a stable subordinator of index $1/2$ over the interval $[0,t]$, conditionally on being $1$ at time $t$. In general, for $t >0$, $\Lambda_{\alpha}^{(t)}({\rm d} z) \coloneqq  z^{-1}p_{z}(-tz)  \mathds{1}_{\{ z >0 \}} {\rm d} z$ is the L\'evy measure of a not killed pure jump subordinator and ${\rm a}_{1}^{(\alpha)}(t) > {\rm a}_{2}^{(\alpha)} (t) > \cdots$ is the ranked jump sizes of this subordinator before time $t$; see \cite{Miermont2001}. We refer to \cite{Perman1992} and \cite[Section 8.1]{PitmanYor1997} for more information about the distribution of the jumps of a subordinator.
 
\section{Stable L\'evy processes, bridges and excursions} \label{Sec2}

In this section, we recall several results about stable L\'evy processes without negative jumps and refer the interesting reader to \cite[Chapter VIII]{Bertoin1996} or the work of Chaumont \cite{Chaumont1997} for further details. \\

\noindent \textbf{Spectrally positive stable L\'evy processes.} Let $(\Omega, \mathcal{F}, \mathbb{P})$ be the underlying probability space. A strictly stable spectrally positive L\'evy process with index $\alpha \in (1,2]$ is a random process $X_{\alpha} = (X_{\alpha}(s), s \geq 0)$ with paths in $\mathbb{D}(\mathbb{R}_{+}, \mathbb{R})$, which has independent and stationary increments, no negative jumps and such that $\mathbb{E}[\exp(-\lambda X_{\alpha}(s))] = \exp(c s  \lambda^{\alpha})$ for every $s, \lambda \geq 0$, and some constant $c >0$. An important feature of $X_{\alpha}$ is the so-called scaling property: for every real constant $k >0$, the process $(k^{-1/\alpha} X_{\alpha}(ks), s \geq 0)$ has the same distribution as $X_{\alpha}$. Then, in this work, we can and we will take $c=1$ if $\alpha = (1,2)$, and $c = 1/2$ if $\alpha =2$, without loss of generality. In particular, for $\alpha=2$, the process $X_{2}$ is the standard Brownian motion on the positive real line. \\

\noindent \textbf{Stable bridge and stable normalized excursion.} The stable L\'evy bridge $X_{\alpha}^{\rm br} = (X_{\alpha}^{\rm br} (s), s \in  [0,1])$ is a random process with paths in $\mathbb{D}([0,1], \mathbb{R})$ that can informally be defined as the process $X_{\alpha}$ conditioned to be at level $0$ at time $1$. This conditioning can be made rigorous and we refer to \cite{Chaumont1997} for details. The normalized excursion $X_{\alpha}^{\rm exc} = (X_{\alpha}^{\rm exc} (s), s \in  [0,1])$ of a spectrally positive $\alpha$-stable L\'evy process with unit lifetime (or  $\alpha$-stable excursion for simplicity) is a random process with paths in $\mathbb{D}([0,1], \mathbb{R})$ that can be thought as the process $X_{\alpha}^{\rm br}$ conditioned to stay nonnegative between times $0$ and $1$. Let us make this more precise and formally define the process $X_{\alpha}^{\rm exc}$. We consider the so-called Vervaat transform (or Vervaat excursion) introduced by Tak\'acs \cite{Tak1967} and used by Vervaat \cite{Veervat1979} to change a bridge type function in $\mathbb{D}([0,1], \mathbb{R})$ into an excursion. More precisely, a bridge is a function $g \in \mathbb{D}([0,1], \mathbb{R})$ such that $g(0) = g(1) = g(1-)=0$. For any $g \in \mathbb{D}([0,1], \mathbb{R})$, we set $\bar{\mu}(g) \coloneqq \inf \{s \in [0,1]: g(s-)\wedge g(s) = \inf_{u \in [0,1]}g(u) \}$, i.e., the smallest location of the infimum of $g$. Then, we define the Vervaat transform $\mathbf{V}$ of a bridge $g \in \mathbb{D}([0,1], \mathbb{R})$ by
\begin{eqnarray*}
\mathbf{V}(g)(s) \coloneqq \left\{ \begin{array}{lcl}
              g(s + \bar{\mu}(g)) - \inf_{u \in [0,1]}g(u) & \mbox{  if } & s \leq 1- \bar{\mu}(g),\\
               g(s + \bar{\mu}(g)-1) - \inf_{u \in [0,1]}g(u)  & \mbox{  if } & s \geq 1- \bar{\mu}(g). \\
              \end{array}
    \right.
\end{eqnarray*}

\noindent Clearly, $\mathbf{V}(g)$ is a path in $\mathbb{D}([0,1], \mathbb{R})$ which only takes nonnegative values and $\mathbf{V}(g)(0) = \mathbf{V}(g)(1)=0$. It is well-known that a stable bridge $X_{\alpha}^{\rm br}$ satisfies $X_{\alpha}^{\rm br}(0) = X_{\alpha}^{\rm br}(1) = X_{\alpha}^{\rm br}(1-)=0$. Moreover, $X_{\alpha}^{\rm br}$ reaches its infimum at a unique random time that $\bar{\mu}_{\alpha} \coloneqq \bar{\mu}(X_{\alpha}^{\rm br})$; see \cite{Chaumont1997}. Thus, we formally define the $\alpha$-stable excursion as the Vervaat transform of the stable bridge $X_{\alpha}^{\rm br}$, i.e., $X_{\alpha}^{\rm exc} \coloneqq \mathbf{V}(X_{\alpha}^{\rm br})$. We refer to the work of Chaumont \cite{Chaumont1997} (see also \cite[Chapter VIII]{Bertoin1996}) for other constructions
of the process $X_{\alpha}^{\rm exc}$ via path transformations, or alternatively, using arguments from excursion theory of Markov processes. A useful property (see \cite[Theorem 4]{Chaumont1997}) that one can deduce from the above construction is that 
\begin{eqnarray} \label{eq17}
\bar{\mu}_{\alpha} \hspace*{3mm} \text{and} \hspace*{3mm}  X_{\alpha}^{\rm exc} \hspace*{2mm} \text{are independent} \hspace*{2mm}  \text{and} \hspace*{2mm} \bar{\mu}_{\alpha} \hspace*{2mm} \text{is uniformly distributed on}  \hspace*{2mm} [0,1].
\end{eqnarray}

\section{The coding of Galton--Watson trees and their fragmentation} \label{Sec3}

In this section, we formally introduce the family of critical Galton--Watson trees and explain how they can be coded by different functions, namely the so-called Łukasiewicz path and a similar path derived by the Prim's algorithm. The latter provides an alternative order on the vertices of the tree, which we refer to as the Prim order. Following \cite{Brou2016}, we will see how the Prim's order of the vertices can be used to define a consistent exploration process of the fragmentation forest that stores all the information of the sizes of its connected components. Finally, we prove a distributional property for this exploration process that will be a crucial ingredient in the proof of Theorem \ref{Theo3}.  \\

\noindent \textbf{Plane trees.} We follow the formalism of Neveu \cite{Ne1986}. Let $\mathbb{N} = \{1, 2, \dots \}$ be the set of positive integers, set $\mathbb{N}^{0} = \{ \varnothing  \}$ and consider the set of labels
$\mathbb{U} = \bigcup_{n \geq 0} \mathbb{N}^{n}$. For $u = (u_{1}, \dots, u_{n}) \in \mathbb{U}$, we denote by $|u| = n $ the length (or generation, or height) of $u$; if $v = (v_{1}, \dots, v_{m}) \in \mathbb{U}$, we let $uv = (u_{1}, \dots, u_{n}, v_{1}, \dots, v_{m}) \in \mathbb{U}$ be the concatenation of $u$ and $v$. A plane tree is a nonempty, finite subset $\tau \subset  \mathbb{U}$ such that: (i) $\varnothing \in \tau$; (ii) if $v \in \tau$ and $v = uj$ for some $j \in \mathbb{N}$, then $u \in \tau$; (iii) if $u \in \tau$, then there exists an integer $c(u) \geq 0$ such that $ui \in  \tau$ if and only if $1 \leq i \leq c(u)$.  We will view each vertex $u$ of a tree $\tau$ as an individual of a population whose genealogical tree is $\tau$. The vertex $\varnothing$ is called the root of the tree and for every $u \in \tau$, $c(u)$ is the number of children of $u$ (if $c(u) = 0$, then $u$ is called a leaf, otherwise, $u$ is called an internal vertex). The total progeny (or size) of $\tau$ will be denoted by $\zeta(\tau) = \text{Card}(\tau)$ (i.e., the number of vertices of $\tau$). We denote by $\mathbb{T}$ the set of plane trees and for each $n \in \mathbb{N}$, by $\mathbb{T}_{n}$ the set of plane trees with $n$ vertices, or equivalently $n-1$ edges. \\

\noindent \textbf{Galton--Watson trees.} Let $\mu$ be a probability measure on $\mathbb{Z}_{+}$ which satisfies $\mu(0)>0$,  expectation $\sum_{k=0}^{\infty} k \mu(k) = 1$ and such that $\mu(0)+\mu(1)<1$. The law of a critical Galton--Watson tree with offspring distribution $\mu$ is the unique probability measure $\mathbb{P}_{\mu}$ on $\mathbb{T}$ satisfying: (i) $\mathbb{P}_{\mu} (c(\varnothing ) = k) = \mu(k)$ for every $k \geq 0$; (ii) For every $k \geq 1$ such that $\mu(k) >0$, conditioned on the event $\{c(\varnothing ) = k \}$, the subtrees that stem from the children of the root $\{ u \in \mathbb{U} : 1u \in \tau \}, \dots,  \{ u \in \mathbb{U} : ku \in \tau \}$ are independent and distributed as $\mathbb{P}_{\mu}$. Otter \cite{Otter1949} shows that the law $\mathbb{P}_{\mu}$ is given by the explicit formula $\mathbb{P}_{\mu}(\tau) = \prod_{u \in \tau} \mu(c(u))$. A random tree whose distribution is $\mathbb{P}_{\mu}$ will be called a Galton--Watson tree with offspring distribution $\mu$. We also denote by $\mathbb{P}_{\mu}^{(n)}$ the law on $\mathbb{T}_{n}$ of a Galton--Watson tree with
offspring distribution $\mu$ conditioned to have $n$ vertices, providing that this conditioning makes sense. \\

\noindent \textbf{Coding planar trees by discrete paths.}
In this work, we will use two different orderings of the vertices of a tree $\tau \in \mathbb{T}$:
\begin{itemize}
\item[(i)]\textbf{Lexicographical ordering.} Given $v,w \in \tau$, we write $v \prec_{ \text{lex}} w$ if there exists $z \in \tau$ such that $v = z(v_{1}, \dots, v_{n})$, $w = z(w_{1}, \dots, w_{m})$ and $v_{1} < w_{1}$. 

\item[(ii)] \textbf{Prim ordering.} Let $\textbf{edge}(\tau)$ be the set of edges of $\tau$ and consider a sequence of distinct and positive weights $\mathbf{w} = (w_{e}: e \in \textbf{edge}(\tau))$ (i.e., each edge $e$ of $\tau$ is marked with a different and positive weight $w_{e}$). If there is an edge connecting two vertices, say $u$ and $v$, in $\tau$, we denote it by $\{u, v\}$. Let us describe the Prim order $\prec_{\text{prim}}$ of the vertices in $\tau$, that is, $\varnothing = u(0) \prec_{\text{prim}}  u(1) \prec_{\text{prim}} \dots  \prec_{\text{prim}} u(\zeta(\tau)-1)$. We will use the notation $V_{i}$ for the set $\{u(0), \dots, u(i-1)\}$, for $1 \leq i \leq \zeta(\tau)$. First set $u(0) = \varnothing$ and $V_{0} = \{u(0)\}$. Suppose that for some $1 \leq i \leq \zeta(\tau)-1$, the vertices $u(0), \dots, u(i-1)$ have been defined. Consider the weights $\{w_{\{u,v\}}: u \in V_{i}, v \not \in V_{i}\}$ of edges between a vertex of $V_{i}$ and another outside of $V_{i}$. Since all the weights are distinct, the minimum weight in $\{w_{\{u,v\}}: u \in V_{i}, v \not \in V_{i}\}$ is reached at an edge $\{\tilde{u}, \tilde{v} \}$ where $\tilde{u} \in V_{i}$ and $\tilde{v} \not \in V_{i}$. Then set $u(i) = \tilde{v}$. This iterative procedure completely determines the Prim order $\prec_{\text{prim}}$.
\end{itemize}

\noindent For $\ast \in \{ {\rm lex}, {\rm prim} \}$, we associate to the ordering $\varnothing = u(0) \prec_{\ast} u(1) \prec_{\ast} \dots \prec_{\ast} u(\zeta(\tau)-1)$ of the vertices of $\tau$ a path $W^{\ast} = (W^{\ast}(k), 0 \leq k \leq \zeta(\tau))$, by letting $W^{\ast}(0) = 0$ and for $0 \leq k \leq \zeta(\tau)-1$, $W^{\ast}(k+1) = W^{\ast}(k)  + c(u(k))-1$, where we recall that $c(u(k))$ denotes the number of children of the vertex $u(k) \in \tau$. Observe that  $W^{\ast}(k+1) - W^{\ast}(k) = c(u(k)) -1 \geq -1$ for every $0 \leq k \leq \zeta(\tau)-1$, with equality if and only if $u(k)$ is a leaf of $\tau$. Note also that $W^{\ast}(k) \geq 0$, for every $0 \leq k \leq \zeta(\tau)-1$, but $W^{\ast}(\zeta(\tau)) = -1$. We shall think of such a path as the step function on $[0,\zeta(\tau)]$ given $s \mapsto W^{\ast}(\lfloor s \rfloor)$. The path $W^{\text{lex}}$ is commonly called Łukasiewicz path of $\tau$, and from now on we refer to $W^{\text{prim}}$ as the Prim path; see Figure \ref{Fig2}. See \cite{Leje2005} for more details and properties on the Łukasiewicz path. 

The procedure just described to obtain the Prim ordering is known as Prim's algorithm (or Prim-Jarn\'ik algorithm); see \cite{Prim}. This algorithm associates to any properly weighted graph its unique minimum spanning tree. In practice, one could also consider that $\mathbf{w}$ is a sequence of i.i.d.\ positive random variables such that they are all distinct a.s. See Figure \ref{Fig1} for an illustration of the previous orderings of the vertices in a tree. 
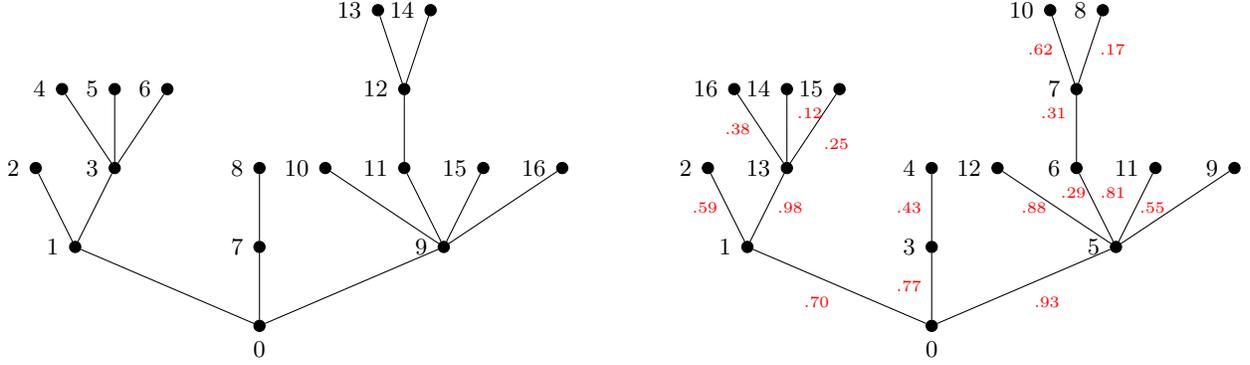
\begin{figure}[!htb]
   \begin{minipage}{0.48\textwidth}
     \centering
      \begin{tikzpicture}[scale=.7,font=\footnotesize]
\tikzstyle{solid node}=[circle,draw,inner sep=1.5,fill=black]
    \tikzstyle{level 1}=[sibling distance=35mm]
    \tikzstyle{level 2}=[sibling distance=15mm]
    \tikzstyle{level 3}=[sibling distance=10mm]
\node[solid node, label=below:{$0$}]{}
[grow=north]
child{node[solid node, label=left:{$9$}]{}
   child{node[solid node, label=left:{$16$}]{}}
   child{node[solid node, label=left:{$15$}]{}}
   child{node[solid node, label=left:{$11$}]{}
     child{node[solid node, label=left:{$12$}]{}
        child{node[solid node, label=left:{$14$}]{}}
        child{node[solid node, label=left:{$13$}]{}}
       }
     }
   child{node[solid node, label=left:{$10$}]{}}
}
child{node[solid node, label=left:{$7$}]{}
   child{node[solid node, label=left:{$8$}]{}}
}
child{node[solid node, label=left:{$1$}]{}
   child{node[solid node, label=left:{$3$}]{}
     child{node[solid node, label=left:{$6$}]{}}
     child{node[solid node, label=left:{$5$}]{}}
     child{node[solid node, label=left:{$4$}]{}}
   }
   child{node[solid node, label=left:{$2$}]{}}
}
;
\end{tikzpicture}
   \end{minipage}\hfill
   \begin{minipage}{0.48\textwidth}
     \centering
      \begin{tikzpicture}[scale=.7,font=\footnotesize]
\tikzstyle{solid node}=[circle,draw,inner sep=1.5,fill=black]
    \tikzstyle{level 1}=[sibling distance=35mm]
    \tikzstyle{level 2}=[sibling distance=15mm]
    \tikzstyle{level 3}=[sibling distance=10mm]
\node[solid node, label=below:{$0$}]{}
[grow=north]
child{node[solid node, label=left:{$5$}]{}
   child{node[solid node, label=left:{$9$}]{} edge from parent node[left,draw=none] {\tiny{\textcolor{red}{$.55$}}}}
   child{node[solid node, label=left:{$11$}]{} edge from parent node[above left,draw=none] {\tiny{\textcolor{red}{$.81$}}} }
   child{node[solid node, label=left:{$6$}]{}
     child{node[solid node, label=left:{$7$}]{}
        child{node[solid node, label=left:{$8$}]{} edge from parent node[right,draw=none] {\tiny{\textcolor{red}{$.17$}}}}
        child{node[solid node, label=left:{$10$}]{} edge from parent node[left,draw=none] {\tiny{\textcolor{red}{$.62$}}}}
       edge from parent node[above left,draw=none] {\tiny{\textcolor{red}{$.31$}}}}
     edge from parent node[above left,draw=none] {\tiny{\textcolor{red}{$.29$}}} }
   child{node[solid node, label=left:{$12$}]{} edge from parent node[left,draw=none] {\tiny{\textcolor{red}{$.88$}}}}
edge from parent node[below right,draw=none] {\tiny{\textcolor{red}{$.93$}}} }
child{node[solid node, label=left:{$3$}]{}
   child{node[solid node, label=left:{$4$}]{} edge from parent node[left,draw=none] {\tiny{\textcolor{red}{$.43$}}}}
edge from parent node[left,draw=none] {\tiny{\textcolor{red}{$.77$}}} }
child{node[solid node, label=left:{$1$}]{}
   child{node[solid node, label=left:{$13$}]{}
     child{node[solid node, label=left:{$15$}]{} edge from parent node[below right,draw=none] {\tiny{\textcolor{red}{$.25$}}}}
     child{node[solid node, label=left:{$14$}]{} edge from parent node[above right,draw=none] {\tiny{\textcolor{red}{$.12$}}}}
     child{node[solid node, label=left:{$16$}]{} edge from parent node[left,draw=none] {\tiny{\textcolor{red}{$.38$}}}}
edge from parent node[right,draw=none] {\tiny{\textcolor{red}{$.98$}}}   
   }
   child{node[solid node, label=left:{$2$}]{} edge from parent node[left,draw=none] {\tiny{\textcolor{red}{$.59$}}}}
edge from parent node[below left,draw=none] {\tiny{\textcolor{red}{$.70$}}}
}
;
\end{tikzpicture}
   \end{minipage}\hfill
\caption{From left to right, a plane tree with vertices labeled in lexicographical order and a weighted plane tree with vertices labeled in Prim order.}\label{Fig1}
\end{figure}

Define the probability measure $\hat{\mu}$ on $\{-1, 0, 1, \dots \}$ by $\hat{\mu}(k) = \mu( k+1 )$ for every $k \geq -1$. Let $X = (X(k), k \geq 0)$ be a random walk which starts at $0$ with jump distribution $\hat{\mu}$ and define also the time $\zeta_{1} = \inf\{k \geq 0: X(k)= -1 \}$. In the Prim ordering, consider that the weights $\mathbf{w}$ are a sequence of i.i.d.\ positive random variables such that they are distinct a.s.
\begin{figure}[!htb]
  \begin{minipage}{0.48\textwidth}
     \centering
\begin{tikzpicture}[scale=.8,font=\footnotesize]
\begin{axis}[axis x line = middle,axis y line = left, title = $W^{\text{lex}}$, width=10cm, height=7cm, xtick={0,2,...,16}, ytick={-1,0,...,4}, xmin=0, xmax=18,  ymin=-1, ymax=4.5]
\addplot[domain=0:1, ultra thick] {0};
\addplot[domain=1:2, ultra thick] {2};
\addplot[domain=2:3, ultra thick] {3};
\addplot[domain=3:4, ultra thick] {2};
\addplot[domain=4:5, ultra thick] {4};
\addplot[domain=5:6, ultra thick] {3};
\addplot[domain=6:7, ultra thick] {2};
\addplot[domain=7:8, ultra thick] {1};
\addplot[domain=8:9, ultra thick] {1};
\addplot[domain=9:10, ultra thick] {0};
\addplot[domain=10:11, ultra thick] {3};
\addplot[domain=11:12, ultra thick] {2};
\addplot[domain=12:13, ultra thick] {2};
\addplot[domain=13:14, ultra thick] {3};
\addplot[domain=14:15, ultra thick] {2};
\addplot[domain=15:16, ultra thick] {1};
\addplot[domain=16:17, ultra thick] {0};
\addplot[domain=17:18, ultra thick] {-1};
\addplot[domain=0:18, dotted] {-1};
\addplot[domain=0:18, dotted] {1};
\addplot[domain=0:18, dotted] {2};
\addplot[domain=0:18, dotted] {3};
\addplot[domain=0:18, dotted] {4};
 \end{axis}
\end{tikzpicture}
   \end{minipage}
  \begin{minipage}{0.48\textwidth}
     \centering
\begin{tikzpicture}[scale=.8,font=\footnotesize]
\begin{axis}[axis x line = middle,axis y line = left, title = $W^{\text{prim}}$, width=10cm, height=7cm, xtick={0,2,...,16}, ytick={-1,0,...,4,5}, xmin=0, xmax=18,  ymin=-1, ymax=5.5]
\addplot[domain=0:1, ultra thick] {0};
\addplot[domain=1:2, ultra thick] {2};
\addplot[domain=2:3, ultra thick] {3};
\addplot[domain=3:4, ultra thick] {2};
\addplot[domain=4:5, ultra thick] {2};
\addplot[domain=5:6, ultra thick] {1};
\addplot[domain=6:7, ultra thick] {4};
\addplot[domain=7:8, ultra thick] {4};
\addplot[domain=8:9, ultra thick] {5};
\addplot[domain=9:10, ultra thick] {4};
\addplot[domain=10:11, ultra thick] {3};
\addplot[domain=11:12, ultra thick] {2};
\addplot[domain=12:13, ultra thick] {1};
\addplot[domain=13:14, ultra thick] {0};
\addplot[domain=14:15, ultra thick] {2};
\addplot[domain=15:16, ultra thick] {1};
\addplot[domain=16:17, ultra thick] {0};
\addplot[domain=17:18, ultra thick] {-1};
\addplot[domain=0:18, dotted] {-1};
\addplot[domain=0:18, dotted] {1};
\addplot[domain=0:18, dotted] {2};
\addplot[domain=0:18, dotted] {3};
\addplot[domain=0:18, dotted] {4};
\addplot[domain=0:18, dotted] {5};
 \end{axis}
\end{tikzpicture}
   \end{minipage}
\caption{In the left, the $\L$ukasiewicz path of the plane tree in Figure \ref{Fig1}. In the right, the Prim path of the plane tree in Figure \ref{Fig1}}\label{Fig2}
\end{figure}
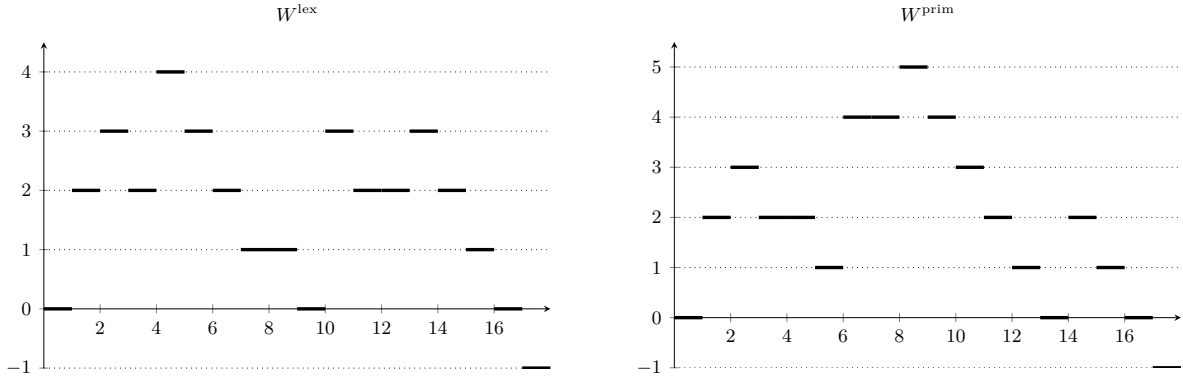

\begin{proposition} \label{Pro1}
For every $\ast \in \{{\rm lex}, {\rm prim} \}$, if we sample a plane tree according to $\mathbb{P}_{\mu}$, then $W^{\ast}$ is distributed as $(X(0), X(1), \dots, X(\zeta_{1}))$. In particular, the total progeny of the sample plane tree has the same distribution as $\zeta_{1}$. 
\end{proposition}
\begin{proof}
The proof for the Łukasiewicz path can be found in \cite[Proposition 1.5]{Leje2005}. For the Prim path the proof follows from a simple adaptation of that of \cite[Proposition 1.5]{Leje2005}; see also \cite[Lemmas 15 and 16]{Brou2016} for an alternative approach.
\end{proof}

\noindent \textbf{Fragmentation of a plane tree.} Consider $\tau \in \mathbb{T}$ and let $\textbf{edge}(\tau)$ denote its set of edges. Equip the edges of $\tau$ with i.i.d.\ uniform random weights $\mathbf{w} = (w_{e}: e \in \textbf{edge}(\tau))$ on $[0,1]$ and independently of the tree $\tau$. In particular, for a vertex $v \in \tau$ with $c(v) \geq 1$ children, we write $(w_{v,k}, 1 \leq k \leq c(v))$ for the weights of the edges connecting $v$ with its children. For $t \in [0,1]$, we then keep the edges of $\tau$ with weight smaller than $t$ and discard the others. This gives rise to a forest $\mathbf{f}_{\tau}(t)$ with the same set of vertices as $\tau$ but with set of edges given by $\textbf{edge}(\mathbf{f}_{\tau}(t)) = \{ e \in \textbf{edge}(\tau): w_{e} \leq t\}$. Furthermore, each vertex $v \in \mathbf{f}_{\tau}(t)$ has $c_{t}(v) = \sum_{k=1}^{c(v)} \mathds{1}_{\{ w_{v,k} \leq t \}}$ children if $c(v) \geq 1$; otherwise, $c_{t}(v) = 0$ whenever $c(v) =0$. In what follows, we refer to the forest $\mathbf{f}_{\tau}(t)$ associated to a plane tree $\tau$ and uniform weights $\mathbf{w}$ as {\sl the fragmented forest} at time $t \in [0,1]$, or simply,  {\sl fragmentation forest}; see Figure \ref{Fig3}. In this work we restrict ourselves to the case uniform i.i.d.\ weights, but certainly some of the forthcoming results can be extended for more general sequences of weights. \\

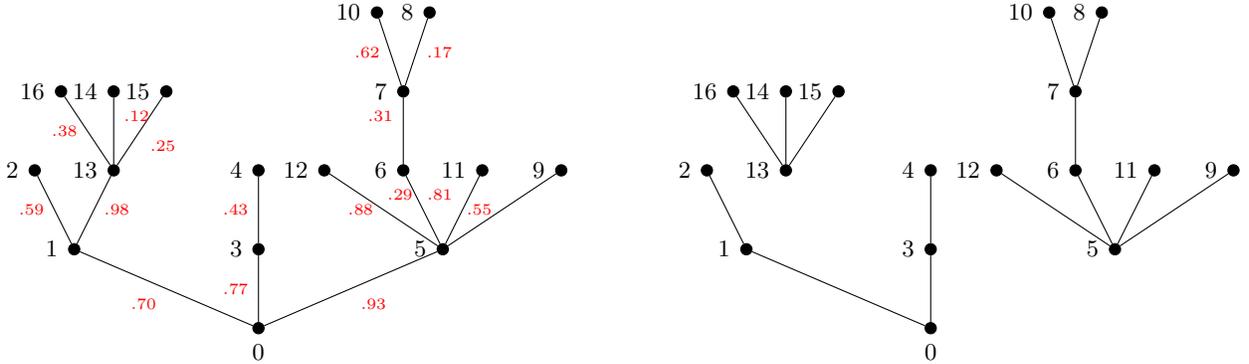
\begin{figure}[!htb]
   \begin{minipage}{0.48\textwidth}
     \centering
      \begin{tikzpicture}[scale=.7,font=\footnotesize]
\tikzstyle{solid node}=[circle,draw,inner sep=1.5,fill=black]
    \tikzstyle{level 1}=[sibling distance=35mm]
    \tikzstyle{level 2}=[sibling distance=15mm]
    \tikzstyle{level 3}=[sibling distance=10mm]
\node[solid node, label=below:{$0$}]{}
[grow=north]
child{node[solid node, label=left:{$5$}]{}
   child{node[solid node, label=left:{$9$}]{} edge from parent node[left,draw=none] {\tiny{\textcolor{red}{$.55$}}}}
   child{node[solid node, label=left:{$11$}]{} edge from parent node[above left,draw=none] {\tiny{\textcolor{red}{$.81$}}} }
   child{node[solid node, label=left:{$6$}]{}
     child{node[solid node, label=left:{$7$}]{}
        child{node[solid node, label=left:{$8$}]{} edge from parent node[right,draw=none] {\tiny{\textcolor{red}{$.17$}}}}
        child{node[solid node, label=left:{$10$}]{} edge from parent node[left,draw=none] {\tiny{\textcolor{red}{$.62$}}}}
       edge from parent node[above left,draw=none] {\tiny{\textcolor{red}{$.31$}}}}
     edge from parent node[above left,draw=none] {\tiny{\textcolor{red}{$.29$}}} }
   child{node[solid node, label=left:{$12$}]{} edge from parent node[left,draw=none] {\tiny{\textcolor{red}{$.88$}}}}
edge from parent node[below right,draw=none] {\tiny{\textcolor{red}{$.93$}}} }
child{node[solid node, label=left:{$3$}]{}
   child{node[solid node, label=left:{$4$}]{} edge from parent node[left,draw=none] {\tiny{\textcolor{red}{$.43$}}}}
edge from parent node[left,draw=none] {\tiny{\textcolor{red}{$.77$}}} }
child{node[solid node, label=left:{$1$}]{}
   child{node[solid node, label=left:{$13$}]{}
     child{node[solid node, label=left:{$15$}]{} edge from parent node[below right,draw=none] {\tiny{\textcolor{red}{$.25$}}}}
     child{node[solid node, label=left:{$14$}]{} edge from parent node[above right,draw=none] {\tiny{\textcolor{red}{$.12$}}}}
     child{node[solid node, label=left:{$16$}]{} edge from parent node[left,draw=none] {\tiny{\textcolor{red}{$.38$}}}}
edge from parent node[right,draw=none] {\tiny{\textcolor{red}{$.98$}}}   
   }
   child{node[solid node, label=left:{$2$}]{} edge from parent node[left,draw=none] {\tiny{\textcolor{red}{$.59$}}}}
edge from parent node[below left,draw=none] {\tiny{\textcolor{red}{$.70$}}}
}
;
\end{tikzpicture}
   \end{minipage}\hfill
   \begin{minipage}{0.48\textwidth}
     \centering
      \begin{tikzpicture}[scale=.7,font=\footnotesize]
\tikzstyle{solid node}=[circle,draw,inner sep=1.5,fill=black]
    \tikzstyle{level 1}=[sibling distance=35mm]
    \tikzstyle{level 2}=[sibling distance=15mm]
    \tikzstyle{level 3}=[sibling distance=10mm]
\node[solid node, label=below:{$0$}]{}
[grow=north]
child{node[solid node, label=left:{$5$}]{}
   child{node[solid node, label=left:{$9$}]{}}
   child{node[solid node, label=left:{$11$}]{}}
   child{node[solid node, label=left:{$6$}]{}
     child{node[solid node, label=left:{$7$}]{}
        child{node[solid node, label=left:{$8$}]{}}
        child{node[solid node, label=left:{$10$}]{}}
       }
     }
   child{node[solid node, label=left:{$12$}]{}}
edge from parent[draw=none] }
child{node[solid node, label=left:{$3$}]{}
   child{node[solid node, label=left:{$4$}]{}}
}
child{node[solid node, label=left:{$1$}]{}
   child{node[solid node, label=left:{$13$}]{}
     child{node[solid node, label=left:{$15$}]{}}
     child{node[solid node, label=left:{$14$}]{}}
     child{node[solid node, label=left:{$16$}]{}}
   edge from parent[draw=none] }
   child{node[solid node, label=left:{$2$}]{}}
}
;
\end{tikzpicture}
   \end{minipage}\hfill
\caption{A plane tree with uniform random weights in the left side. In the right side, the forest created by keeping the edges with weight at most $t =.92$. The vertices are labelled according to the Prim ordering.}\label{Fig3}
\end{figure}

\noindent \textbf{Prim exploration of the fragmentation forest.}  For a plane tree $\tau \in \mathbb{T}$ and sequence of  i.i.d.\ uniform random weights $\mathbf{w}$ on $[0,1]$, let $\mathbf{f}_{\tau}(t)$ be the fragmentation forest of $\tau$ at time $t \in [0,1]$. Let us now explain how to explore the subtree components of the forest $\mathbf{f}_{\tau}(t)$ by using the approach outlined in \cite[page 532]{Brou2016} (see also \cite{Aldous1997}). For $t \in [0,1]$, denote by $\textbf{Neigh}_{t}(v) \coloneqq \{u \in \mathbf{f}_{\tau}(t): \{u,v\} \in \textbf{edge}(\mathbf{f}_{\tau}(t))\}$ the set of neighbours of $v \in \mathbf{f}_{\tau}(t)$. For a set of vertices $V$ of $\mathbf{f}_{\tau}(t)$, let also $\textbf{Neigh}_{t}(V) \coloneqq ( \bigcup_{v \in V} \textbf{Neigh}_{t}(v)) \setminus V$, the set of neighbours of vertices in $V$ but not in $V$. We associate to the prim ordering $\varnothing = u(0) \prec_{{\rm prim}} u(1) \prec_{{\rm prim}} \dots \prec_{{\rm prim}} u(\zeta(\tau)-1)$ of the vertices of $\tau$ the following exploration process of $\mathbf{f}_{\tau}(t)$ (recall that $\mathbf{f}_{\tau}(t)$ and $\tau$ have the same set of vertices). The first visited vertex is $v_{t}(0)=u(0)$. Suppose that we have explored the vertices $V_{k} = \{v_{t}(0), \dots, v_{t}(k-1)\}$ at some time $1 \leq k  \leq \zeta(\tau)$. If $k = \zeta(\tau)$, we have finished the exploration, and otherwise, one has two possibilities:
\begin{itemize}
\item[(i)] if $\textbf{Neigh}_{t}(V_{k}) \neq \varnothing$, then $v_{t}(k)$ is the next vertex according to the order $\prec_{{\rm prim}}$ that belongs to $\textbf{Neigh}_{t}(V_{k})$, or
\item[(ii)] if $\textbf{Neigh}_{t}(V_{k}) = \varnothing$, then $v_{t}(k)$ is the next vertex according to the order $\prec_{{\rm prim}}$ that belongs to $\tau \setminus V_{k}$. 
\end{itemize}

\noindent This exploration process results in an order for the vertices of $\mathbf{f}_{\tau}(t)$ (equivalently, to the vertices of $\tau$) that we denote by $<_{{\rm prim}}$ (i.e.\ $\varnothing = v_{t}(0) <_{{\rm prim}} v_{t}(1) <_{{\rm prim}} \dots <_{{\rm prim}} v_{t}(\zeta(\tau)-1)$) and call {\sl Prim exploration}. An important feature of the Prim exploration of $\mathbf{f}_{\tau}(t)$ is that the Prim ordering $<_{\text{prim}}$ of its vertices is preserved for all values of $t \in [0,1]$. More precisely, for $t_{1}, t_{2} \in [0,1]$, $v_{t_{1}}(k) = v_{t_{2}}(k)$, for all $0 \leq k \leq \zeta(\tau)-1$; see Figure \ref{Fig3} for an example when $t_{1}=1$ and $t_{2}=.92$. This is a consequence of the algorithm to obtain the Prim ordering of the vertices in $\tau$ which associates to any properly weighted graph its unique minimum spanning tree. We henceforth write $\prec_{\text{prim}}$ instead of $<_{\text{prim}}$ and remove the subindex $t$ from our notation, i.e., we write $\varnothing = v(0) \prec_{\text{prim}} v(1) \prec_{\text{prim}} \dots \prec_{\text{prim}} v(\zeta(\tau)-1)$ for the vertices of $\mathbf{f}_{\tau}(t)$ in Prim order, which is the same as the Prim ordering of the vertices of the tree $\tau$, $\varnothing = u(0) \prec_{\text{prim}} u(1) \prec_{\text{prim}} \dots \prec_{\text{prim}} u(\zeta(\tau)-1)$ presented earlier. 

Following the presentation of \cite[pages 532-533]{Brou2016}, one can associate to the Prim ordering of the vertices of $\mathbf{f}_{\tau}(t)$, an {\sl exploration path} $Z_{t} = (Z_{t}(k), 0 \leq k \leq \zeta(\tau) + 1)$ by letting $Z_{t}(0) = Z(\zeta(\tau)+1) = 0$, and for $1 \leq k \leq \zeta(\tau)$, $Z_{t}(k) = \text{Card}(\textbf{Neigh}_{t}(V_{k}))$. Furthermore, let $\textbf{CC}(\mathbf{f}_{\tau}(t))$ be the set of connected components of $\mathbf{f}_{\tau}(t)$. Then \cite[Lemma 14]{Brou2016} shows that
\begin{eqnarray*}
\text{Card}( \{k \in \{1, \dots, \zeta(\tau)\}: Z_{t}(k)=0 \})= \text{Card}(\textbf{CC}(\mathbf{f}_{\tau}(t))),
\end{eqnarray*}

\noindent and that the successive sizes of the connected components ordered by the exploration coincide with the distances between successive $0$'s in the sequence $Z_{t} = (Z_{t}(k), 0 \leq k \leq \zeta(\tau)+1)$; see Figure \ref{Fig4}.
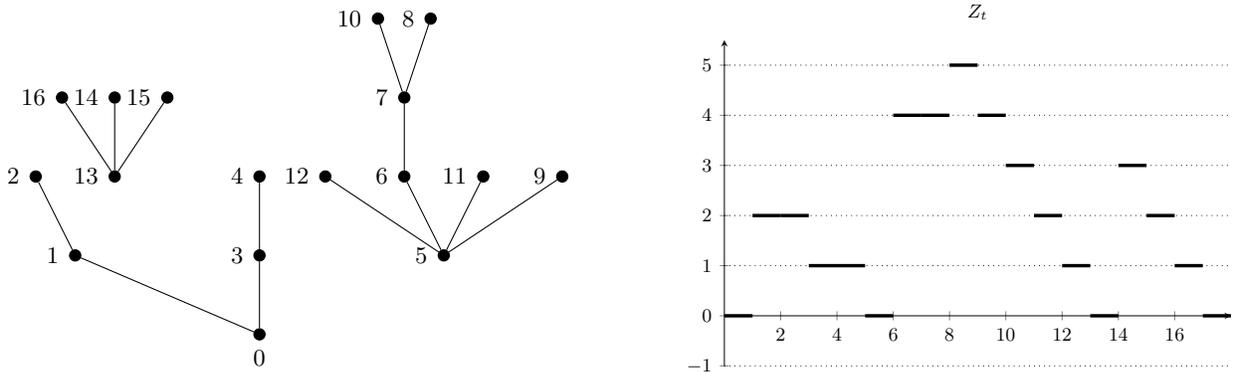
\begin{figure}[!htb]
  \begin{minipage}{0.48\textwidth}
     \centering
      \begin{tikzpicture}[scale=.7,font=\footnotesize]
\tikzstyle{solid node}=[circle,draw,inner sep=1.5,fill=black]
    \tikzstyle{level 1}=[sibling distance=35mm]
    \tikzstyle{level 2}=[sibling distance=15mm]
    \tikzstyle{level 3}=[sibling distance=10mm]
\node[solid node, label=below:{$0$}]{}
[grow=north]
child{node[solid node, label=left:{$5$}]{}
   child{node[solid node, label=left:{$9$}]{}}
   child{node[solid node, label=left:{$11$}]{}}
   child{node[solid node, label=left:{$6$}]{}
     child{node[solid node, label=left:{$7$}]{}
        child{node[solid node, label=left:{$8$}]{}}
        child{node[solid node, label=left:{$10$}]{}}
       }
     }
   child{node[solid node, label=left:{$12$}]{}}
edge from parent[draw=none] }
child{node[solid node, label=left:{$3$}]{}
   child{node[solid node, label=left:{$4$}]{}}
}
child{node[solid node, label=left:{$1$}]{}
   child{node[solid node, label=left:{$13$}]{}
     child{node[solid node, label=left:{$15$}]{}}
     child{node[solid node, label=left:{$14$}]{}}
     child{node[solid node, label=left:{$16$}]{}}
   edge from parent[draw=none] }
   child{node[solid node, label=left:{$2$}]{}}
}
;
\end{tikzpicture}
   \end{minipage}\hfill
  \begin{minipage}{0.48\textwidth}
     \centering
\begin{tikzpicture}[scale=.8,font=\footnotesize]
\begin{axis}[axis x line = middle,axis y line = left, title = $Z_{t}$, width=10cm, height=7cm, xtick={0,2,...,16}, ytick={-1,0,...,4,5}, xmin=0, xmax=18,  ymin=-1, ymax=5.5]
\addplot[domain=0:1, ultra thick] {0};
\addplot[domain=1:2, ultra thick] {2};
\addplot[domain=2:3, ultra thick] {2};
\addplot[domain=3:4, ultra thick] {1};
\addplot[domain=4:5, ultra thick] {1};
\addplot[domain=5:6, ultra thick] {0};
\addplot[domain=6:7, ultra thick] {4};
\addplot[domain=7:8, ultra thick] {4};
\addplot[domain=8:9, ultra thick] {5};
\addplot[domain=9:10, ultra thick] {4};
\addplot[domain=10:11, ultra thick] {3};
\addplot[domain=11:12, ultra thick] {2};
\addplot[domain=12:13, ultra thick] {1};
\addplot[domain=13:14, ultra thick] {0};
\addplot[domain=14:15, ultra thick] {3};
\addplot[domain=15:16, ultra thick] {2};
\addplot[domain=16:17, ultra thick] {1};
\addplot[domain=17:18, ultra thick] {0};
\addplot[domain=0:18, dotted] {-1};
\addplot[domain=0:18, dotted] {1};
\addplot[domain=0:18, dotted] {2};
\addplot[domain=0:18, dotted] {3};
\addplot[domain=0:18, dotted] {4};
\addplot[domain=0:18, dotted] {5};
 \end{axis}
\end{tikzpicture}
   \end{minipage}
\caption{In the left side, the forest of Figure \ref{Fig3}. In the right side, its exploration path $Z_{t}$. The vertices are labelled according to the Prim ordering.}\label{Fig4}
\end{figure}

In this work, and in analogy with the coding paths of $\tau$ introduced earlier, we will consider a slight modification of the exploration path $Z_{t}$. More precisely, define the Prim path $W^{\rm prim}_{t} = (W_{t}^{\rm prim}(k), 0 \leq k \leq \zeta(\tau))$ by letting $W_{t}^{\rm prim}(0) = 0$, and for $0 \leq k \leq \zeta(\tau)-1$, $W_{t}^{\rm prim}(k+1) = W_{t}^{\rm prim}(k) + c_{t}(v_{t}(k))-1$, where $c_{t}(v)$ denotes the number of children of $v \in \mathbf{f}_{\tau}(t)$. We shall also think of such a path as the step function on $[0,\zeta(\tau)]$ given by $s \mapsto W_{t}^{\rm prim}(\lfloor s \rfloor)$. 
\begin{lemma} \label{lemma4}
Let $\tau \in \mathbb{T}$ and $\mathbf{w}$ be a sequence of i.i.d.\ uniform random weights on $[0,1]$. For any time $t \in [0,1]$,
\begin{eqnarray*}
{\rm Card} \left( \left \{k \in \{1, \dots, \zeta(\tau)\}: W_{t}^{\rm prim}(k)= \min_{0 \leq m \leq k} W_{t}^{\rm prim}(m) \right \} \right)= {\rm Card}({\bf CC}(\mathbf{f}_{\tau}(t))), 
\end{eqnarray*}

\noindent Moreover, the successive sizes of the connected components of $\mathbf{f}_{\tau}(t)$ ordered by the exploration process coincide with the distances between successive new minimums in the sequence $(W_{t}^{\rm prim}(k), 0 \leq k \leq \zeta(\tau))$.
\end{lemma}

\begin{proof}
The result is an immediate consequence of the previous discussion. 
\end{proof}

Indeed, the sizes of the connected components of $\mathbf{f}_{\tau}(t)$ coincides with the length of the excursions of the walk $W_{t}^{\rm prim}$ above its minimum; see Figure \ref{Fig5}. 

\begin{figure}[!htb]
  \begin{minipage}{0.48\textwidth}
     \centering
      \begin{tikzpicture}[scale=.7,font=\footnotesize]
\tikzstyle{solid node}=[circle,draw,inner sep=1.5,fill=black]
    \tikzstyle{level 1}=[sibling distance=35mm]
    \tikzstyle{level 2}=[sibling distance=15mm]
    \tikzstyle{level 3}=[sibling distance=10mm]
\node[solid node, label=below:{$0$}]{}
[grow=north]
child{node[solid node, label=left:{$5$}]{}
   child{node[solid node, label=left:{$9$}]{}}
   child{node[solid node, label=left:{$11$}]{}}
   child{node[solid node, label=left:{$6$}]{}
     child{node[solid node, label=left:{$7$}]{}
        child{node[solid node, label=left:{$8$}]{}}
        child{node[solid node, label=left:{$10$}]{}}
       }
     }
   child{node[solid node, label=left:{$12$}]{}}
edge from parent[draw=none] }
child{node[solid node, label=left:{$3$}]{}
   child{node[solid node, label=left:{$4$}]{}}
}
child{node[solid node, label=left:{$1$}]{}
   child{node[solid node, label=left:{$13$}]{}
     child{node[solid node, label=left:{$15$}]{}}
     child{node[solid node, label=left:{$14$}]{}}
     child{node[solid node, label=left:{$16$}]{}}
   edge from parent[draw=none] }
   child{node[solid node, label=left:{$2$}]{}}
}
;
\end{tikzpicture}
   \end{minipage}\hfill
  \begin{minipage}{0.48\textwidth}
     \centering
\begin{tikzpicture}[scale=.8,font=\footnotesize]
\begin{axis}[axis x line = middle,axis y line = left, title = $W_{t}^{\text{prim}}$, width=10cm, height=7cm, xtick={0,2,...,16}, ytick={-3,-2,-1,0,1,2,3}, xmin=0, xmax=18,  ymin=-3, ymax=3.5]
\addplot[domain=0:1, ultra thick] {0};
\addplot[domain=1:2, ultra thick] {1};
\addplot[domain=2:3, ultra thick] {1};
\addplot[domain=3:4, ultra thick] {0};
\addplot[domain=4:5, ultra thick] {0};
\addplot[domain=5:6, ultra thick] {-1};
\addplot[domain=6:7, ultra thick] {2};
\addplot[domain=7:8, ultra thick] {2};
\addplot[domain=8:9, ultra thick] {3};
\addplot[domain=9:10, ultra thick] {2};
\addplot[domain=10:11, ultra thick] {1};
\addplot[domain=11:12, ultra thick] {0};
\addplot[domain=12:13, ultra thick] {-1};
\addplot[domain=13:14, ultra thick] {-2};
\addplot[domain=14:15, ultra thick] {0};
\addplot[domain=15:16, ultra thick] {-1};
\addplot[domain=16:17, ultra thick] {-2};
\addplot[domain=17:18, ultra thick] {-3};
\addplot[domain=0:18, dotted] {2};
\addplot[domain=0:18, dotted] {1};
\addplot[domain=0:18, dotted] {-1};
\addplot[domain=0:18, dotted] {-2};
\addplot[domain=0:18, dotted] {-3};
 \end{axis}
\end{tikzpicture}
   \end{minipage}
\caption{In the left side, the forest of Figure \ref{Fig3} with vertices labelled according to the Prime ordering. In the right side, its Prim path $W_{t}^{\text{prim}}$. }\label{Fig5}
\end{figure}
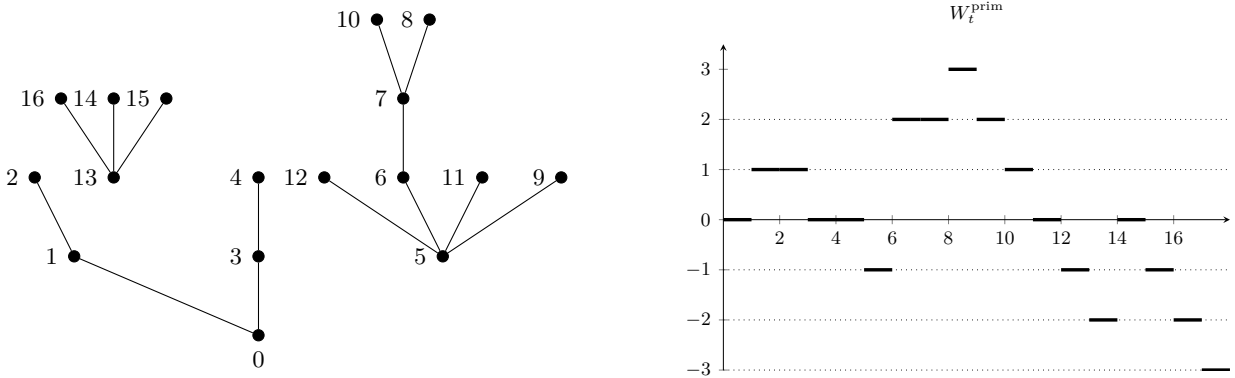

Following Proposition \ref{Pro1}, the Prim path of the fragmentation forest associated to a critical Galton--Watson tree with offspring distribution $\mu$ can also be related to a random walk. Recall that $X = (X(k), k \geq 0)$ denotes a random walk that starts at $0$ and has jump distribution $\hat{\mu}$ on $\{-1,0,1, \dots\}$. Recall also that we write $\zeta_{1} = \inf\{k \geq 0: X(k)= -1 \}$. Denote by ${\bm \xi} = (\xi(k), k \geq 1)$ the increments of $X$, i.e.\ $\xi(k) = X(k)-X(k-1)$, for $k\geq 1$. Let $(U_{k}(j))_{k,j \geq 1}$ be a sequence of i.i.d.\ uniform random variables on $[0,1]$. For $t \in [0,1]$, define ${\bm \xi}_{t} = (\xi_{t}(k), k \geq 1)$ by letting 
\begin{eqnarray*}
\xi_{t}(k) = \sum_{j=1}^{\xi(k)+1} \mathds{1}_{\{ U_{k}(j) \leq t \}}, \hspace*{3mm} \text{for} \hspace*{2mm} t \in [0,1], \hspace*{2mm} k \geq 1,
\end{eqnarray*}

\noindent with the convention $\sum_{j=1}^{0} \mathds{1}_{\{ U_{k}(j) \leq t \}} = 0$. Hence, $\xi_{0}(k) = 0$, $\xi_{1}(k) = \xi(k)+1$ and for any $k \geq 1$, the mapping $t \mapsto \xi_{t}(k)$ is non-decreasing. Let $X_{t} = (X_{t}(k), k \geq 0)$ be the process defined by
\begin{eqnarray} \label{NRwalk}
X_{t}(0)=0 \hspace*{3mm} \text{and} \hspace*{3mm} X_{t}(k) = \sum_{i=1}^{k} (\xi_{t}(i)-1), \hspace*{3mm} \text{for} \hspace*{2mm} t \in [0,1], \hspace*{2mm} k \geq 1.
\end{eqnarray}
\begin{proposition} \label{Pro2}
Sample a plane tree $\mathbf{t}$ according to $\mathbb{P}_{\mu}$, i.e., consider a critical Galton--Watson tree $\mathbf{t}$ with offspring $\mu$. Let $\mathbf{w} = (w_{e}: e \in {\bf edge}(\mathbf{t}))$ be a sequence of i.i.d.\ uniform random weights on $[0,1]$. Then, the Prim path $W_{t}^{{\rm prim}}$ satisfies
\begin{eqnarray*}
(W_{t}^{{\rm prim}}(0), W_{t}^{{\rm prim}}(1), \dots, W_{t}^{{\rm prim}}(\zeta(\mathbf{t})))_{t \in [0,1]} \stackrel{d}{=} (X_{t}(0), X_{t}(1), \dots, X_{t}(\zeta_{1}))_{t \in [0,1]},
\end{eqnarray*}
\noindent where $\stackrel{d}{=}$ means equality of finite-dimensional distributions. 
\end{proposition}
\begin{proof}
For $t \in [0,1]$, we write $V(0) = \varnothing, V(1), \dots, V(\zeta(\mathbf{t})-1)$ for the vertices of $\mathbf{f}_{\mathbf{t}}(t)$ listed in Prim order. To simplify the notation, for $t \in [0,1]$ and $k = 0, \dots, \zeta(\tau)-1$, we will write $c_{t}(V(k)) = c_{t}(k)$ for the number of children of the vertex $V(k)$ in $\mathbf{f}_{\mathbf{t}}(t)$. Recall that $\mathbf{f}_{\mathbf{t}}(1) = \mathbf{t}$ and $c_{1}(V(k)) = c(V(k))$. In particular, we will write $c(k) = c(V(k))$. To prove our claim, it is enough to check that 
\begin{eqnarray*}
(c_{t}(0), c_{t}(1), \dots, c_{t}(\zeta(\mathbf{t})-1))_{t \in [0,1]} \stackrel{d}{=} (\xi_{t}(1), \dots, \xi_{t}(\zeta_{1}))_{t \in [0,1]},
\end{eqnarray*}

\noindent in the sense of finite-dimensional distributions. 

Consider the infinite tree $\mathbb{U}$ and denote by $\textbf{edge}(\mathbb{U})$ its set of edges. Denote by $\textbf{Neigh}(v_{0}) \coloneqq \{u \in \mathbb{U}: \{u,v_{0}\} \in \textbf{edge}(\mathbb{U})\}$ the set of neighbours of $v_{0} \in \mathbb{U}$. For $r \in \mathbb{N}$ and a set of vertices $S_{r} \coloneqq \{v_{0}, v_{1}, \dots, v_{r-1}\}$ of $\mathbb{U}$, we also write $\textbf{Neigh}(S_{r}) \coloneqq ( \bigcup_{v \in S_{r}} \textbf{Neigh}(v)) \setminus S_{r}$ for the set of neighbours of vertices in $S_{r}$ but not in $S_{r}$. For $v_{0} = \varnothing, v_{1} \in \textbf{Neigh}(v_{0}), \dots, v_{r} \in \textbf{Neigh}(\{v_{0}, v_{1}, \dots, v_{r-1}\})$, define the event 
\begin{eqnarray*}
\textbf{N}(v_{0}, v_{1}, \dots, v_{r-1}) \coloneqq \{V(0) = v_{0}, V(1) = v_{1}, \dots, V(\zeta(\mathbf{t}) -1) = v_{r} \} \cap \{ \zeta(\mathbf{t}) = r \}. 
\end{eqnarray*}

\noindent For $r \in \mathbb{N}$ and $k_{0}, k_{1}, \dots, k_{r-1} \in \mathbb{N} \cup \{0\}$, we also define the event
\begin{eqnarray*}
\textbf{C}(k_{0}, k_{1}, \dots, k_{r-1}) \coloneqq \{c(0) = k_{0}, c(1) = k_{1}, \dots, c(\zeta(\mathbf{t}) -1) = k_{r} \} \cap \{ \zeta(\mathbf{t}) = r \}. 
\end{eqnarray*}

\noindent For simplicity, given a measurable set $A$, we write $\mathbb{E}[ \cdot; A] = \mathbb{E}[ \cdot \mathds{1}_{A}]$, and given a finite collection of measurable sets $A_{1}, \dots, A_{i}$, we shall write $\mathbb{E}[ \cdot; A_{1}, \dots, A_{i}] = \mathbb{E}[ \cdot \mathds{1}_{A_{i} \cap \dots \cap A_{i}}]$, for $i \in \mathbb{N}$. For fixed $n \in \mathbb{N}$, we set $0 \leq t_{1} \leq \cdots \leq t_{n} \leq 1$, and for $i =1,2, \dots, n$ and $r \in \mathbb{N}$, consider $g_{0}^{i}, g_{1}^{i}, \dots g_{r-1}^{i}$ nonnegative functions on $\{0,1,\dots \}$. Hence,
\begin{align*}
& \mathbb{E}\left[ \prod_{i=1}^{n} g_{0}^{i}(c_{t_{i}}(0))  g_{1}^{i}(c_{t_{i}}(1))\cdots g_{r-1}^{i}(c_{t_{i}}(\zeta(\mathbf{t})-1)); \textbf{N}(v_{0}, v_{1}, \dots, v_{r-1}),  \textbf{C}(k_{0}, k_{1}, \dots, k_{r-1}),  \zeta(\mathbf{t}) = r \right]  \\
& \hspace*{3mm} = \mathbb{E}\left[ \prod_{i=1}^{n} g_{0}^{i}(c_{t_{i}}(v_{0}))  g_{1}^{i}(c_{t_{i}}(v_{1}))\cdots g_{r-1}^{i}(c_{t_{i}}(v_{r-1})); \textbf{N}(v_{0}, v_{1}, \dots, v_{r-1}),  \textbf{C}(k_{0}, k_{1}, \dots, k_{r-1}),  \zeta(\mathbf{t}) = r \right].
\end{align*}

\noindent For $t \in [0,1]$ and $p = 0, \dots, \zeta(\tau)-1$, recall that if $c(p) \geq 1$, then $c_{t}(p) = \sum_{i=1}^{c(p)} \mathds{1}_{\{ w_{V(p),i} \leq t \}}$. Otherwise, $c_{t}(p) = 0$ whenever $c(p)=0$. Then, in the event $\textbf{N}(v_{0}, v_{1}, \dots, v_{r-1}) \cap  \textbf{C}(k_{0}, k_{1}, \dots, k_{r-1}) \cap \{ \zeta(\mathbf{t}) = r \}$, we have that $c_{t}(v_{p}) = \sum_{i=1}^{k_{p}} \mathds{1}_{\{ w_{v_{p},i} \leq t \}}$; with the convention that the sum is equal to zero if it is empty. Define the random variables, $\kappa_{t}(p) =  \sum_{j=1}^{k_{p}} \mathds{1}_{\{ U_{p+1}(j) \leq t \}}$; with the convention that the sum is equal to zero whenever it is empty. We then see that 
\begin{align*}
& \mathbb{E}\left[ \prod_{i=1}^{n} g_{0}^{i}(c_{t_{i}}(0))  g_{1}^{i}(c_{t_{i}}(1))\cdots g_{r-1}^{i}(c_{t_{i}}(\zeta(\mathbf{t})-1)); \textbf{N}(v_{0}, v_{1}, \dots, v_{r-1}),  \textbf{C}(k_{0}, k_{1}, \dots, k_{r-1}),  \zeta(\mathbf{t}) = r \right]  \\
& \hspace*{3mm} = \prod_{p=0}^{r-1}\mathbb{E}\left[ g_{p}^{1}(\kappa_{t_{1}}(p)) \cdots g_{p}^{n}(\kappa_{t_{n}}(p)) \right] \mathbb{P}( \textbf{N}(v_{0}, v_{1}, \dots, v_{r-1}) \cap \textbf{C}(k_{0}, k_{1}, \dots, k_{r-1}) \cap  \zeta(\mathbf{t}) = r ).
\end{align*}

\noindent Therefore, by summing over all possible, $k_{p}$'s and $v_{p}$'s, Proposition \ref{Pro1}
implies that
\begin{eqnarray*}
\mathbb{E}\left[ \prod_{i=1}^{n} g_{0}^{i}(c_{t_{i}}(0))  g_{1}^{i}(c_{t_{i}}(1))\cdots g_{r -1}^{i}(c_{t_{i}}(r -1)); \zeta(\mathbf{t}) = r \right] = \prod_{k=0}^{r -1} \mathbb{E}\left[ g_{k}^{1}(\xi_{t_{1}}(k+1)) \cdots g_{k}^{n}(\xi_{t_{n}}(k+1)); \zeta_{1} = r \right],
\end{eqnarray*}

\noindent which concludes our proof. 
\end{proof}

\section{Convergence of the exploration processes} \label{Sec4}

Recall that $\mathbb{P}^{(n)}_{\mu}$ denotes the law of a critical Galton--Watson tree with offspring distribution $\mu$ conditioned to have $n \in \mathbb{N}$ vertices. For every $n \in \mathbb{N}$, for which $\mathbb{P}^{(n)}_{\mu}$ is well-defined, sample a plane tree on $\mathbb{T}_{n}$, say $\mathbf{t}_{n}$, according to $\mathbb{P}^{(n)}_{\mu}$, i.e., $\mathbf{t}_{n}$ is a critical Galton--Watson tree conditioned to have $n$ vertices. Through this section we assume that $\mu$ belongs to the domain of attraction of a stable law of index $\alpha \in (1,2]$, and refer to $\mathbf{t}_{n}$ as an $\alpha$-stable $\text{GW}$-tree. We will always let $\mathbf{w} = (w_{e}: e \in {\bf edge}(\mathbf{t}_{n}))$ be a sequence of i.i.d.\ uniform random weights on $[0,1]$. We write $W_{n}^{\rm lex} = (W_{n}^{\rm lex}( \lfloor n u \rfloor ), u \in [0,1])$ for the associated time-scaled Łukasiewicz path of $\mathbf{t}_{n}$. We also write  $W_{n}^{\rm prim} = (W_{n}^{\rm prim}( \lfloor n u \rfloor ), u \in [0,1])$ for the time-scaled Prim path of $\mathbf{t}_{n}$ with respect to $\mathbf{w}$.

The asymptotic behavior of large $\alpha$-stable $\text{GW}$-trees is well understood, in particular through scaling limits of their associated Łukasiewicz paths; see, e.g., \cite{Du2003}. In this section, we first show that the Prim path of $\mathbf{t}_{n}$ has the same asymptotic behavior as its Łukasiewicz path. Then, we use this as a stepping stone to study the Prim path of the fragmentation forest of $\mathbf{t}_{n}$ associated to the weights $\mathbf{w}$. Recall that $X_{\alpha}^{\rm exc} = (X_{\alpha}^{\rm exc}(u), u \in [0,1] )$ denotes the $\alpha$-stable excursion of index $\alpha$; see Section \ref{Sec2}.
\begin{theorem} \label{Theo1}
Let $\mathbf{t}_{n}$ be an $\alpha$-stable ${\rm GW}$-tree, and let $(B_{n})_{n \geq 1}$ be a sequence of positive real numbers satisfying (\ref{eq10}). For $\ast \in \{ {\rm lex}, {\rm prim} \}$, we have that
\begin{eqnarray*} 
\left( \frac{1}{B_{n}} W_{n}^{\ast}(\lfloor nu \rfloor) , u \in [0,1]  \right) \xrightarrow[ ]{d} ( X_{\alpha}^{\rm exc}(u), u \in [0,1] ), \hspace*{3mm} \text{as} \hspace*{2mm}  n \rightarrow \infty, \hspace*{2mm} \text{in the space} \hspace*{2mm} \mathbb{D}([0,1], \mathbb{R}).
\end{eqnarray*}
\end{theorem}
\begin{proof}
The proof for the Łukasiewicz path can be found in \cite[Theorem 3.1]{Du2003}. For the Prim path the result follows from that of the Łukasiewicz path and Proposition \ref{Pro1}. 
\end{proof}

For $s \in [0,1]$, let $\mathbf{f}_{n}(s)$ be the fragmentation forest of $\mathbf{t}_{n}$ at time $s$. Denote by $W_{n,s}^{\rm prim} = (W_{n,s}^{\rm prim}(\lfloor nu \rfloor), u \in [0,1])$ the time-scaled Prim path of $\mathbf{f}_{n}(s)$. In particular, $W_{n,1}^{\rm prim}$ is exactly $W_{n}^{\rm prim}$. For fixed $t \geq 0$, consider the sequence $(s_{n}(t))_{n \geq 1}$ of positive times given by
\begin{eqnarray*}
s_{n}(t) = \max \left(1-\frac{B_{n}}{n}t, 0 \right),
\end{eqnarray*} 

\noindent where $(B_{n})_{n \geq 1}$ a sequence of positive real numbers satisfying (\ref{eq10}). Define the process $W_{n}^{(t)} = (W_{n}^{(t)}(u),u \in [0,1])$ by letting
\begin{eqnarray} \label{eq13}
W_{n}^{(t)}(u) =  \frac{1}{B_{n}} W_{n,s_{n}(t)}^{\rm prim}(\lfloor nu \rfloor), \hspace*{4mm} \text{for} \hspace*{2mm} u \in [0,1]. 
\end{eqnarray}

\noindent Later, in the proof of Theorem \ref{Theo3}, we will refer to the process $W_{n}^{(t)}$ as the (normalized and time-scaled) Prim path of the fragmentation forest at time $s_{n}(t)$, i.e., $\mathbf{f}(s_{n}(t))$. We then set $W_{n} = (W_{n}^{(t)}, t \geq 0)$. From the previous section, the mapping $t \mapsto W_{n}^{(t)}$ is non-increasing in $t$ which implies that the process $W_{n}$ has c\`adl\`ag paths. Thus, we will view $(t, u) \mapsto W_{n}^{(t)}(u) $ as a random variable taking values in the space $\mathbb{D}(\mathbb{R}_{+}, \mathbb{D}([0,1], \mathbb{R}))$ of $\mathbb{D}([0,1], \mathbb{R})$-valued c\`adl\`ag functions on $\mathbb{R}$ equipped with the Skorokhod topology. In other words, for fixed $t \geq 0$, $W_{n}^{(t)}$ is a random variable in $\mathbb{D}([0,1], \mathbb{R})$.

We introduce the continuous counterpart of the process $W_{n}$.  For every $t \geq 0$, let $Y_{\alpha}^{(t)} = (Y_{\alpha}^{(t)}(u), u \in [0,1])$ be defined by $Y_{\alpha}^{(t)}(u) = X_{\alpha}^{\rm exc}(u) - tu$, for $u \in [0,1]$. In particular, for $t=0$, $Y_{\alpha}^{(0)} = X_{\alpha}^{\rm exc}$ and we sometimes write $X_{\alpha}^{\rm exc}$ instead of $Y_{\alpha}^{(0)}$, for simplicity. Then, define the process $Y_{\alpha} = (Y_{\alpha}^{(t)}, t \geq 0)$.

The following theorem is the main result of this section.
\begin{theorem} \label{Theo2}
We have the convergence 
\begin{eqnarray*} 
(W_{n}^{(t)}, t \geq 0) \xrightarrow[ ]{d} (Y_{\alpha}^{(t)}, t \geq 0), \hspace*{3mm} \text{as} \hspace*{2mm} n \rightarrow \infty, \hspace*{2mm} \text{in the space} \hspace*{2mm} \mathbb{D}(\mathbb{R}_{+}, \mathbb{D}([0,1], \mathbb{R})).
\end{eqnarray*}
\end{theorem}

Theorem \ref{Theo2} generalizes \cite[Theorem 10]{Brou2016}. Specifically, in \cite{Brou2016}, the authors only consider the case when $\mathbf{t}_{n}$ is a ${\rm GW}$-tree with $\mu$ being the law of a Poisson random variable of parameter $1$ (i.e., $\mathbf{t}_{n}$ is a Cayley tree) while our setting is clearly more general. As in most proofs for convergence of stochastic processes, the proof of Theorem \ref{Theo2} consists in two steps: convergence of the finite-dimensional distributions and tightness of the sequence of processes $(W_{n})_{n \geq 1}$. To accomplish the above, recall the random walk connected to the Prim path of the fragmentation forest of the $\alpha$-stable ${\rm GW}$-tree $\mathbf{t}_{n}$ (Proposition \ref{Pro2}). More precisely, for $s \in [0,1]$, let $X_{s} = (X_{s}(k), k \geq 0)$ be the stochastic process defined in (\ref{NRwalk}). For $n \in \mathbb{N}$ and $t \geq 0$, define the process $Y_{n}^{(t)} = (Y_{n}^{(t)}(u), u \in [0,1])$ by letting 
\begin{eqnarray*}
Y_{n}^{(t)}(u) =  \frac{1}{B_{n}} X_{s_{n}(t)}(\lfloor nu \rfloor), \hspace*{4mm} \text{for} \hspace*{2mm} u \in [0,1],
\end{eqnarray*}

\noindent and set $Y_{n} = (Y_{n}^{(t)}, t \geq 0)$. From Proposition \ref{Pro2}, we see that $W_{n}$ has the same finite-dimensional distribution as $Y_{n}$ under the conditional probability distribution $\mathbb{P}_{n}(\cdot) \coloneqq \mathbb{P}( \cdot | \zeta_{1} = n)$. In the following, we will always work with the process $Y_{n}$ (or $Y_{n}^{(t)}$) under the conditional probability distribution $\mathbb{P}_{n}$, and to keep the notation simple, we will continue to write $Y_{n}$ (and $Y_{n}^{(t)}$) also for the conditional version. \\

\noindent \textbf{Finite-dimensional distributions.} We start with two observations that will be used quite often. Proposition \ref{Pro1} and Theorem \ref{Theo1} imply that 
\begin{eqnarray} \label{eq1}
\left( Y_{n}^{(0)}(u) , u \in [0,1] \right) \xrightarrow[ ]{d} \left( X_{\alpha}^{\rm exc}(u), u \in [0,1]  \right), \hspace*{3mm} \text{as} \hspace*{2mm}  n \rightarrow \infty, \hspace*{2mm} \text{in the space} \hspace*{2mm} \mathbb{D}([0,1], \mathbb{R}).
\end{eqnarray}

\noindent For $g \in \mathbb{D}([0,1], \mathbb{R})$, we write $\Vert g \Vert_{\infty} \coloneqq \sup_{u \in [0,1]} |g(u)|$.  Since the supremum is a continuous functional on $\mathbb{D}([0,1], \mathbb{R})$ (see e.g.\ \cite[Proposition 2.4 in Chapter VI]{jacod2003}), (\ref{eq1}) implies that
\begin{eqnarray} \label{eq12}
\Vert Y_{n}^{(0)} \Vert_{\infty} \xrightarrow[ ]{d} \Vert X_{\alpha}^{\rm exc} \Vert_{\infty}, \hspace*{2mm} \text{as} \hspace*{2mm}  n \rightarrow \infty, \hspace*{2mm} \text{in distribution and} \hspace*{2mm} \Vert X_{\alpha}^{\rm exc} \Vert_{\infty} < \infty \hspace*{2mm} \text{a.s.}
\end{eqnarray}

We continue with the convergence of the finite-dimensional distributions. 
\begin{lemma} \label{lemma1}
For $k, m \in \mathbb{N}$, and for any $u_{1}, \dots, u_{k} \in [0,1]$ and $t_{1}, \dots, t_{m} \in \mathbb{R}_{+}$, we have that
\begin{eqnarray*}
\left( Y_{n}^{(t_{i})}(u_{r}): 0 \leq r \leq k, 0 \leq i \leq m  \right) \xrightarrow[ ]{d}  \left( Y_{\alpha}^{(t_{i})}(u_{r}): 0 \leq r \leq k, 0 \leq i \leq m  \right), \hspace*{2mm} \text{as} \hspace*{2mm}  n \rightarrow \infty. 
\end{eqnarray*} 
\end{lemma}
\begin{proof}
By the Skorokhod representation theorem, we can assume that (\ref{eq1}) and (\ref{eq12}) hold almost surely.  For $u \in [0,1]$ and $t \geq 0$, we have that
\begin{eqnarray*}
Y^{(t)}_{n}(u) =  \frac{1}{B_{n}} \sum_{k=1}^{\lfloor nu \rfloor} (\xi_{s_{n}(t)}(k)-1) = \frac{1}{B_{n}} \sum_{k=1}^{\lfloor nu \rfloor} \Big(-1 + \sum_{i=1}^{\xi(k)+1} \mathds{1}_{\{U_{k}(i) \leq 1-t B_{n}/n \}} \Big).
\end{eqnarray*}

\noindent Since $\sum_{k=1}^{\lfloor nu \rfloor} (\xi(k)+1) =  \lfloor nu \rfloor + B_{n} Y^{(0)}_{n}(u)$, we see that 
\begin{eqnarray} \label{eq2}
Y^{(t)}_{n}(u) = S^{(t)}_{n}(u) - \frac{1}{B_{n}} \lfloor nu \rfloor + \frac{1}{B_{n}}\left(1 - \frac{B_{n}}{n} t \right) (\lfloor nu \rfloor + B_{n} Y^{(0)}_{n}(u)),
\end{eqnarray}

\noindent where we set
\begin{eqnarray} \label{eq3}
S^{(t)}_{n}(u) = \frac{1}{B_{n}} \sum_{k=1}^{\lfloor nu \rfloor}  \sum_{i=1}^{\xi(k)+1} \left( \mathds{1}_{\{U_{k}(i) \leq 1-tB_{n}/n \}} - \left(1 - \frac{B_{n}}{n}t \right) \right).
\end{eqnarray}

For fixed $t \geq 0$, the terms in the sum (\ref{eq3}) are independent centred random variables whose variance is bounded by $t B_{n}/n$. Moreover, these terms are also independent of $(\xi(k), 1 \leq k \leq n)$. Since the number of summands in the sum (\ref{eq3}) is bounded by 
\begin{eqnarray*}
 \sum_{k=1}^{\lfloor nu \rfloor}  \sum_{i=1}^{\xi(k)+1} 1 = \sum_{k=1}^{\lfloor nu \rfloor} (\xi(k)+1) =  \lfloor nu \rfloor + B_{n} Y^{(0)}_{n}(u) \leq n + B_{n}\Vert Y_{n}^{(0)} \Vert_{\infty},
\end{eqnarray*}
\noindent Chebyshev's inequality together with (\ref{eq12}) implies that $S^{(t)}_{n}(u) \rightarrow 0$, as $n \rightarrow \infty$, in probability. For the remaining terms at the right-hand side of (\ref{eq2}), we see that (\ref{eq1}) implies that 
\begin{eqnarray*}
 - \frac{1}{B_{n}} \lfloor nu \rfloor + \frac{1}{B_{n}}\left(1 - \frac{B_{n}}{n}t \right) (\lfloor nu \rfloor + B_{n} Y^{(0)}_{n}(u)) \xrightarrow[ ]{ } X_{\alpha}^{\rm exc}(u)- tu, \hspace*{2mm} \text{as} \hspace*{2mm} n \rightarrow \infty,
\end{eqnarray*}

\noindent almost surely. Finally, for any $u \in [0,1]$ and $t \geq 0$, $Y^{(t)}_{n}(u) \xrightarrow[ ]{ } Y_{\alpha}^{(t)}(u)$, 
as $n \rightarrow \infty$, in probability, which implies our claim. 
\end{proof}

\noindent \textbf{Tightness.} Since we are going to work with processes with sample paths in the set $\mathbb{D}(\mathbb{R}_{+}, \mathbb{D}([0,1], \mathbb{R}))$ equipped with the Skorokhod topology, we start by recalling some aspects of this space of c\`adl\`ag functions and refer to \cite[Chapter 3]{Billi1999} (or \cite[Chapter VI]{jacod2003}) for details. Fix a separable, complete metric space $(\mathbb{M}, d)$, and consider the space $\mathbb{D}(\mathbb{R}_{+}, \mathbb{M})$ of c\`adl\`ag functions from $\mathbb{R}_{+}$ to $\mathbb{M}$. For $a >0$, $0 < \delta <1$ and $k \in \mathbb{N}$, a sequence $\Delta_{a,k} = \{0 = t_{0} < t_{1} < \cdots < t_{k} = a \}$ of subdivisions of $[0, a]$ is called $\delta$-sparse if it satisfies $\min_{1 \leq i \leq k} (t_{i} - t_{i-1}) \geq \delta$. The so-called modified modulus of continuity in $\mathbb{D}(\mathbb{R}_{+}, \mathbb{M})$ is given by
\begin{eqnarray*}
\tilde{\omega}(\delta, a, d; g) \coloneqq \inf_{\Delta_{a,k}} \, \max_{1 \leq i \leq k} \, \, \sup_{r,r^{\prime} \in [t_{i-1}, t_{i})} d(g(r), g(r^{\prime})),  \hspace*{3mm} \text{for} \hspace*{2mm} g \in \mathbb{D}(\mathbb{R}_{+}, \mathbb{M}),
\end{eqnarray*}

\noindent where the infimum extends over all $\delta$-sparse sets $\Delta_{a,k}$. Let $\Theta$ denote the class of strictly increasing, continuous mappings of $[0,1]$ onto itself. For $\theta \in \Theta$, we put 
\begin{eqnarray*}
\Vert \theta \Vert^{\circ} \coloneqq \sup_{0 \leq r < r^{\prime} \leq 1} \left| \log \frac{\theta(r^{\prime}) - \theta(r)}{r^{\prime}-r} \right|,
\end{eqnarray*}

\noindent and recall that the Skorokhod metric in $\mathbb{D}([0,1], \mathbb{M})$ is defined by
\begin{eqnarray*}
{\rm Sk}_{d}(g, g^{\prime}) \coloneqq \inf_{\theta \in \Theta} \left\{\Vert \theta \Vert^{\circ} \vee \sup_{0 \leq r \leq 1} d(g(r),g^{\prime}(\theta(r))) \right\}, \hspace*{3mm} \text{for} \hspace*{2mm} g, g^{\prime} \in \mathbb{D}([0,1], \mathbb{M}),
\end{eqnarray*}

\noindent where the infimum extends over all $\theta \in \Theta$ such that $\Vert \theta \Vert^{\circ} < \infty$ and $\sup_{0 \leq r \leq 1} d(g(r),g^{\prime}(\theta(r))) < \infty$. It is well-known that the metric space $(\mathbb{D}([0,1], \mathbb{M}), {\rm Sk}_{d})$ is complete and separable; see \cite[Theorem 12.2, Chapter 3]{Billi1999}. In particular, if $\mathbb{M} = \mathbb{R}$, we will consider the separable and complete metric space $(\mathbb{R}, | \cdot |)$, where $| \cdot |$ is the Euclidean metric. 
\begin{lemma} \label{lemma2}
For any $a >0$ and $\varepsilon, \varepsilon^{\prime} >0$, there exists $0 < \delta < 1$ such that  
\begin{eqnarray} \label{eq4}
\limsup_{n \rightarrow \infty}\mathbb{P}_{n}(\tilde{\omega}(\delta, a, {\rm Sk}_{| \cdot |}; Y_{n}) \geq \varepsilon) \leq \varepsilon^{\prime}.
\end{eqnarray}

\noindent In particular, the sequence of stochastic processes $(Y_{n})_{n \geq 1}$ is tight on $\mathbb{D}(\mathbb{R}_{+}, \mathbb{D}([0,1], \mathbb{R}))$. 
\end{lemma}

As a preparation for the proof of Lemma \ref{lemma2}, we need a technical result. For $t \geq 0$ and  $g^{(t)} = (g^{(t)}(u), u \in [0,1]) \in \mathbb{D}([0,1], \mathbb{R})$, recall that we write $\Vert g^{(t)} \Vert_{\infty} = \sup_{u \in [0,1]} |g^{(t)}(u)|$. Then for $g = (g^{(t)}, t \geq 0) \in \mathbb{D}(\mathbb{R}_{+}, \mathbb{D}([0,1], \mathbb{R}))$, $a >0$ and $0 < \delta < 1$, define the modulus 
\begin{eqnarray*}
\omega(\delta, a; g) \coloneqq \sup \{ \Vert g^{(t)} - g^{(t^{\prime})} \Vert_{\infty}: |t - t^{\prime}| < \delta, \, \, 0 \leq t, t^{\prime} \leq a \}.
\end{eqnarray*}

\noindent For $t \geq 0$, let $S^{(t)}_{n} = (S^{(t)}_{n}(u), u \in [0,1])$ be the process defined in (\ref{eq3}) and set $S_{n} = (S^{(t)}_{n}, t \geq 0)$.
\begin{lemma} \label{lemma3}
For any $a >0$ and $\varepsilon, \varepsilon^{\prime} >0$, there exists $0 < \delta < 1$ such that  
\begin{eqnarray*}
\limsup_{n \rightarrow \infty} \mathbb{P}_{n}(\omega(\delta, a; S_{n}) \geq \varepsilon) \leq \varepsilon^{\prime}.
\end{eqnarray*}
\end{lemma}

We postpone the proof of Lemma \ref{lemma3} for later and continue with the proof of Lemma \ref{lemma2}. 

\begin{proof}[Proof of Lemma \ref{lemma2}]
Suppose that we have proven (\ref{eq4}) in Lemma \ref{lemma2}. Thanks to the arbitrariness of $\varepsilon, \varepsilon^{\prime} >0$, one can see that for each $a >0$,
\begin{eqnarray*}
\lim_{\delta \rightarrow 0} \limsup_{n \rightarrow \infty} \mathbb{E} [\tilde{\omega}(\delta, a, {\rm Sk}_{| \cdot |}; Y_{n}) \wedge 1 ] = 0.
\end{eqnarray*}

\noindent Then, \cite[Theorem 16.10, Chapter 16]{Kall2005} and Lemma \ref{lemma1} show that the sequence of processes $(Y_{n})_{n \geq 1}$ is tight on $\mathbb{D}(\mathbb{R}_{+}, \mathbb{D}([0,1], \mathbb{R}))$. So, it suffices to prove (\ref{eq4}) to finish the proof of Lemma \ref{lemma2}. 

Fix $a >0$, and observe from (\ref{eq2}) that for $0 \leq t_{1} < t_{2} \leq a$ and $u \in [0,1]$,
\begin{eqnarray*}
Y^{(t_{1})}_{n}(u) - Y^{(t_{2})}_{n}(u) = S^{(t_{1})}_{n}(u) - S^{(t_{2})}_{n}(u) + \frac{\lfloor nu \rfloor}{n}(t_{2}-t_{1}) + \frac{B_{n}}{n}(t_{2}-t_{1})Y^{(0)}_{n}(u).
\end{eqnarray*}

\noindent Since the identity map on $[0,1]$ belongs to $\Theta$, the triangle inequality implies that
\begin{eqnarray*}
{\rm Sk}_{| \cdot |}(Y^{(t_{1})}_{n}, Y^{(t_{2})}_{n}) & = & \inf_{\theta \in \Theta} \left\{\Vert \theta \Vert^{\circ} \vee \sup_{u \in [0,1]} |Y^{(t_{1})}_{n}(u) - Y^{(t_{2})}_{n}(\theta(u))| \right\} \\
& \leq & \Vert S^{(t_{1})}_{n} - S^{(t_{2})}_{n} \Vert_{\infty} + (t_{2}- t_{1}) + a \frac{B_{n}}{n} \Vert Y^{(0)}_{n} \Vert_{\infty}.
\end{eqnarray*}

\noindent For the set $[0,a)$ and each $0 < \delta < \min \{a/2,1/2 \}$, we can have a $\delta$-sparse set $\Delta_{a,k}$ satisfying $\delta \leq t_{i} - t_{i-1} \leq 2 \delta$, for $1 \leq i \leq k$. Then, 
\begin{eqnarray*}
\tilde{\omega}(\delta, a, {\rm Sk}_{| \cdot |}; Y_{n}) \leq \omega(2 \delta, a; S_{n}) + 2 \delta +  2\delta \frac{B_{n}}{n} \Vert Y^{(0)}_{n} \Vert_{\infty}, \hspace*{3mm} \text{for} \hspace*{2mm} 0 < \delta < \min \{a/2,1/2 \}.
\end{eqnarray*}

\noindent Then, (\ref{eq4}) follows from the previous inequality, the convergence in (\ref{eq12}) and Lemma \ref{lemma3}.
\end{proof}

\begin{proof}[Proof of Lemma \ref{lemma3}]
For $n \in \mathbb{N}$ and $a >0$, set $t_{0} = 0$ and $t_{r} = ra/\lceil B_{n} \rceil$, for $r =1, \dots, \lceil B_{n} \rceil$. For $0 \leq r < r^{\prime} \leq \lceil B_{n} \rceil$, define the process $Z_{r, r^{\prime}, n} = (Z_{r, r^{\prime}, n}(u), u \in [0,1])$ by letting 
\begin{eqnarray*}
Z_{r,  r^{\prime}, n}(u) \coloneqq S^{(t_{r})}_{n}(u) - S^{(t_{r^{\prime}})}_{n}(u) = \frac{1}{B_{n}} \sum_{k=1}^{\lfloor nu \rfloor}  \sum_{i=1}^{\xi(k)+1} \left( \mathds{1}_{ \left\{1 - \frac{B_{n}}{n} t_{r^{\prime}}  < U_{k}(i) \leq 1-\frac{B_{n}}{n} t_{r} \right\}} - \frac{B_{n}}{n} (t_{r^{\prime}} - t_{r}) \right).
\end{eqnarray*}

\noindent Recall that $\sum_{k=1}^{\lfloor nu \rfloor} (\xi(k)+1) =  \lfloor nu \rfloor + B_{n} Y^{(0)}_{n}(u)$. For $r=0, 1, \dots, \lceil B_{n} \rceil-1$ and $t_{r} \leq t \leq t_{r+1}$, 
\begin{eqnarray*}
\left| S^{(t_{r})}_{n}(u) - S^{(t)}_{n}(u) \right| & \leq & \left| S^{(t_{r})}_{n}(u) - S^{(t_{r+1})}_{n}(u) \right| + \frac{a}{\lceil B_{n} \rceil n} (\lfloor nu \rfloor + B_{n} Y^{(0)}_{n}(u))  \\
& \leq & \left| Z_{r, r^{\prime}, n}(u) \right| +  \frac{a}{\lceil B_{n} \rceil}  +  \frac{1}{n}\Vert Y^{(0)}_{n} \Vert_{\infty}.
\end{eqnarray*}

\noindent For $n$ large enough, the triangle inequality together with the previous inequality implies that  
\begin{align} \label{eq6}
\omega(\delta, a; S_{n}) & \leq 2\sup \{ \Vert S^{(t_{r})}_{n} - S^{(t)}_{n} \Vert_{\infty} : \, \, 0 \leq r \leq \lceil B_{n} \rceil -1, \, \, t_{r} \leq  t \leq t_{r+1} \} \nonumber \\
& \hspace*{10mm} + \sup \{ \Vert S^{(t_{r})}_{n} - S^{(t_{r^{\prime}})}_{n} \Vert_{\infty} : \, \, 0 \leq r < r^{\prime} \leq \lceil B_{n} \rceil, \, \, |t_{r} - t_{r^{\prime}}| < \delta \} \nonumber \\
& \leq 2 \sup \{ \Vert Z_{r, r+1, n}  \Vert_{\infty} : \, \, 0 \leq r \leq \lceil B_{n} \rceil -1 \} +   \frac{2 a}{\lceil B_{n} \rceil}  +  \frac{2}{n}\Vert Y^{(0)}_{n} \Vert_{\infty} \nonumber \\
& \hspace*{10mm} +  \sup \{ \Vert Z_{r, r^{\prime}, n} \Vert_{\infty} : \, \, 0 \leq r < r^{\prime} \leq \lceil B_{n} \rceil, \, \, |t_{r} - t_{r^{\prime}}| < \delta \}.
\end{align}

\noindent for $0 < \delta < 1$. We will prove that for all $\varepsilon >0$, there is a constant $C_{\varepsilon, p} >0$ such that for all $p \geq 2$, 
\begin{eqnarray} \label{eq5}
\mathbb{P}_{n}\left(\Vert Z_{r,r^{\prime}, n} \Vert_{\infty} \geq \varepsilon \right) \leq C_{\varepsilon, p} B_{n}^{-p/2} (1-B_{n}/n)^{p/2} (t_{r^{\prime}} - t_{r})^{p/2}, \hspace*{4mm} 0 \leq r < r^{\prime} \leq \lceil B_{n} \rceil. 
\end{eqnarray}

\noindent Then, Lemma \ref{lemma3} will follow from (\ref{eq12}), (\ref{eq6}) and the union bound. 

Observe that $\Vert Z_{r, r^{\prime}, n} \Vert_{\infty} = \sup_{1 \leq m \leq n} \left| Z_{r,r^{\prime} n}(m/n) \right|$. By Etemadi's inequality, we have that 
\begin{eqnarray} \label{eq5b}
\mathbb{P}_{n}\left(\Vert Z_{r, r^{\prime}, n} \Vert_{\infty} \geq \varepsilon \right) \leq 3 \sup_{1 \leq m \leq n} \mathbb{P}_{n}\left(\left| Z_{r, r^{\prime}, n}(m/n) \right| \geq \varepsilon /3 \right),
\end{eqnarray}

\noindent for all $\varepsilon >0$. On the one hand, the terms in the sum $Z_{r, r^{\prime}, n}(m/n)$ are independent centred random variables with variance bounded by $a/n$. On the other hand, these terms are also independent of the random variables $(\xi(k), 1 \leq k \leq n)$. Moreover, the number of summands in the sum $Z_{r, r^{\prime}, n}(m/n)$ is bounded by 
\begin{eqnarray*}
 \sum_{k=1}^{\lfloor nu \rfloor}  \sum_{i=1}^{\xi(k)+1} 1 \leq \sum_{k=1}^{n} (\xi(k)+1) =  n + B_{n} Y^{(0)}_{n}(1) = n - B_{n},
\end{eqnarray*}
\noindent under $\mathbb{P}_{n}$. By the Marcinkiewicz–Zygmund inequality, it is not difficult to see that, for $p \geq 2$,
\begin{eqnarray*}
\mathbb{E} [\left| Z_{r, r^{\prime}, n}(m/n) \right|^{p}] \leq C_{p} B_{n}^{-p/2} (1-B_{n}/n)^{p/2} (t_{r^{\prime}} - t_{r})^{p/2}, 
\end{eqnarray*}

\noindent for some constant $C_{p} >0$. So, (\ref{eq5}) follows from (\ref{eq5b}) and Chebyshev's inequality. 
\end{proof}

We have now all the ingredients to prove Theorem \ref{Theo2}.

\begin{proof}[Proof of Theorem \ref{Theo2}]
Theorem \ref{Theo2} is a consequence of Proposition \ref{Pro2}, Lemma \ref{lemma1} and Lemma \ref{lemma2}.
\end{proof}

\section{Proof of Theorem \ref{Theo3}} \label{Sec7}

In this section, we prove Theorem \ref{Theo3}. We start by developing a general approach for the convergence of fragmentation processes encoded by functions in $\mathbb{D}([0,1], \mathbb{R})$. Recall that $\mathbb{S}$ denotes the space defined in (\ref{eq21}) endowed with the $\ell^{1}$-norm. For an increasing function $h = (h(s), s \in [0,1]) \in \mathbb{D}([0,1], \mathbb{R})$, we write 
\begin{eqnarray*}
\mathbf{F}(h) \coloneqq (F_{1}(h), F_{2}(h), \dots  ) \in  \mathbb{S}
\end{eqnarray*}

\noindent for the sequence of the lengths of the intervals components of the complement of the support of the Stieltjes measure ${\rm d} h$, arranged in decreasing order; we tacitly understand $\mathbf{F}(h)$ as an infinite sequence, by completing with an infinite number of zero terms. Let $\text{Supp}({\rm d} h)$ denote the support of ${\rm d} h$ and note that $(0,1) \setminus \text{Supp}({\rm d} h)$ is the union of all open intervals on each of which the function $h$ is constant. For any $g = (g(s), s \in [0,1]) \in \mathbb{D}([0,1], \mathbb{R})$ such that $g(0) = 0$, let $\hat{g} = (\hat{g}(s), s \in [0,1])$ be given by
\begin{eqnarray*}
\hat{g}(s) \coloneqq \inf_{u \in [0, s]} g(u), \hspace*{4mm} s \in [0,1].
\end{eqnarray*}

\noindent Note that $-\hat{g}(s) = \sup_{u \in [0, s]} (-g(u))$, then $-\hat{g}$ is an increasing function in $\mathbb{D}([0,1], \mathbb{R})$. In particular, the Stieltjes measure ${\rm d} (-\hat{g})$ is well-defined and $\text{Supp}({\rm d} (-\hat{g}))$ is given by the set of points where the function $g$ reaches a new infimum. We call constancy interval of $-\hat{g}$ any interval component of $(0,1) \setminus \text{Supp}({\rm d} (-\hat{g}))$. Indeed, those constancy intervals corresponds to excursion intervals of $g$ above its infimum (or equivalently, excursion intervals of the function $g - \hat{g}$ above $0$). 
 
For $g = (g^{(t)}, t \geq 0) \in \mathbb{D}(\mathbb{R}_{+}, \mathbb{D}([0,1], \mathbb{R}))$, we let $g^{(t)} = (g^{(t)}(s), s \in [0,1]) \in  \mathbb{D}([0,1], \mathbb{R})$, for $t \geq 0$. Similarly, for $n \in \mathbb{N}$, we write $g_{n} = (g_{n}^{(t)}, t \geq 0) \in \mathbb{D}(\mathbb{R}_{+}, \mathbb{D}([0,1], \mathbb{R}))$ such that, for each $t \geq 0$, $g_{n}^{(t)} = (g_{n}^{(t)}(s), s \in [0,1]) \in  \mathbb{D}([0,1], \mathbb{R})$. If $g^{(t)}(0) = 0$ (resp.\ $g_{n}^{(t)}(0) = 0)$, we define $\hat{g}^{(t)} = (\hat{g}^{(t)}(s), s\in [0,1])$ (resp.\ $\hat{g}^{(t)}_{n} = (\hat{g}_{n}^{(t)}(s), s \in [0,1])$) by letting 
\begin{eqnarray*}
\hat{g}^{(t)}(s) \coloneqq \inf_{u \in [0, s]} g^{(t)}(u) \hspace*{3mm} \left(\text{resp.} \, \, \hat{g}^{(t)}_{n}(s) \coloneqq \inf_{u \in [0, s]} g^{(t)}_{n}(u) \right), \hspace*{4mm} s \in [0,1].
\end{eqnarray*}

The following result is the key ingredient in the proof of Theorem \ref{Theo3}. Recall that $\mathbb{S}_{1} \subset \mathbb{S}$ denotes the space of the elements of $\mathbb{S}$ with sum $1$.

\begin{lemma} \label{lemma5}
On some probability space $(\Omega, \mathcal{F}, \mathbb{P})$, let $(g_{n})_{n \geq 1}$ be a sequence of random elements of $\mathbb{D}(\mathbb{R}_{+}, \mathbb{D}([0,1], \mathbb{R}))$ such that $g_{n}^{(t)}(0) = 0$, for $n \in \mathbb{N}$ and $t \geq 0$. Suppose that there exists a random $g \in \mathbb{D}(\mathbb{R}_{+}, \mathbb{D}([0,1], \mathbb{R}))$ such that $g^{(t)}(0) = 0$, for $t \geq 0$, and 
\begin{itemize}
\item[(i)] $g_{n} \xrightarrow[]{d} g$, as $n \rightarrow \infty$, in the space $\mathbb{D}(\mathbb{R}_{+}, \mathbb{D}([0,1], \mathbb{R}))$.
\end{itemize}
\noindent Suppose further that for every fixed $t \geq 0$,
\begin{itemize}
\item[(ii)] $g^{(t)}(s) \wedge g^{(t)}(s-) > \hat{g}^{(t)}(s)$, for every $s \in (a,b)$ whenever $(a,b) \subset [0,1]$ is an interval of constancy for the function $-\hat{g}^{(t)}$. 

\item[(iii)] $\mathbf{F}(-\hat{g}^{(t)}) \in \mathbb{S}_{1}$,
\end{itemize}

\noindent where (ii) and (iii) hold almost surely. Then,
\begin{eqnarray*}
(\mathbf{F}(- \hat{g}_{n}^{(t)}), t \geq 0) \rightarrow (\mathbf{F}(-\hat{g}^{(t)}), t\geq 0), \hspace*{3mm} \text{as} \hspace*{2mm}  n \rightarrow \infty, 
\end{eqnarray*}
\noindent in the sense of weak convergence of finite-dimensional distributions in $\mathbb{D}(\mathbb{R}_{+}, \mathbb{S})$. 
\end{lemma}

\begin{proof}
By the Skorokhod representation theorem, we can and we will work in a probability space where the convergence in (i) together with (ii) and (iii) holds almost surely. By (i), there exists a dense subset $D$ of $\mathbb{R}_{+}$ such that for any fixed $k \in \mathbb{N}$ and collection $0 \leq t_{1} < t_{2} < \cdots < t_{k} < \infty$ with $t_{1}, \dots, t_{k} \in D$, we have that a.s.,
\begin{eqnarray*}
(g_{n}^{(t_{1})}, \dots, g_{n}^{(t_{k})}) \rightarrow (g^{(t_{1})}, \dots, g^{(t_{k})}), \hspace*{3mm} \text{as} \hspace*{3mm} n \rightarrow \infty,
\end{eqnarray*}

\noindent in $\mathbb{D}([0,1], \mathbb{R})^{\otimes k}$ (i.e., the $k$-fold space of $\mathbb{D}([0,1], \mathbb{R})$). Then \cite[Lemma 4]{Bertoin2001} implies that a.s.,
\begin{eqnarray*}
(\mathbf{F}(- \hat{g}_{n}^{(t_{1})}), \dots, \mathbf{F}(- \hat{g}_{n}^{(t_{k})})) \rightarrow (\mathbf{F}(-\hat{g}^{(t_{1})}), \dots, \mathbf{F}(-\hat{g}^{(t_{k})})), \hspace*{3mm} \text{as} \hspace*{3mm} n \rightarrow \infty,
\end{eqnarray*}

\noindent in $\mathbb{S}^{\otimes k}$ (i.e., the $k$-fold space of $\mathbb{S}$ equipped with the $\ell^{1}$-norm). Note that the conditions in \cite[Lemma 4]{Bertoin2001} are satisfied by our assumptions (in fact, one has to apply \cite[Lemma 4]{Bertoin2001} to $-g_{n}$ and $-g$).  This shows the convergence of the finite-dimensional distributions of the sequence of processes $((\mathbf{F}(- \hat{g}_{n}^{(t)}), t \geq 0))_{n \geq 1}$ to those of the process $(\mathbf{F}(- \hat{g}^{(t)}), t \geq 0)$. 
\end{proof}

Finally, we are in position to prove our main result Theorem \ref{Theo3}. 
 
\begin{proof}[Proof of Theorem \ref{Theo3}]
Let $\mathbf{t}_{n}$ be an $\alpha$-stable $\text{GW}$-tree of index $\alpha \in (1,2]$. Recall that $(B_{n})_{n \geq 1}$ denotes a sequence of positive real numbers satisfying (\ref{eq10}). For $t \geq 0$, let $W_{n}^{(t)}$ be the (normalized and time-scaled) Prim path defined in (\ref{eq13}) of the fragmentation forest at time $s_{n}(t) = 1 - (B_{n}/n)t$, i.e.\ $\mathbf{f}(s_{n}(t))$, associated to $\mathbf{t}_{n}$ and the i.i.d.\ uniform random weights $\mathbf{w}$. Define the process $I_{n}^{(t)} = (I_{n}^{(t)}(u), u \in [0,1])$ by letting
\begin{eqnarray*}
I_{n}^{(t)}(u) = \inf_{s \in [0,u]} W_{n}^{(t)}(s),  \hspace*{4mm} \text{for} \hspace*{2mm} s \in [0,1].
\end{eqnarray*}

\noindent Recall that $\mathbf{F}_{n}^{(\alpha)} = (\mathbf{F}_{n}^{(\alpha)}(t), t \geq 0)$ stands for the fragmentation process of $\mathbf{t}_{n}$ defined in (\ref{eq15}). From Lemma \ref{lemma4} and the preceding discussion, it is clear that $\mathbf{F}_{n}^{(\alpha)}(t) = \mathbf{F}( - I_{n}^{(t)})$, for $t \geq 0$. Let $Y_{\alpha}^{(t)}$ and $I_{\alpha}^{(t)}$ be the processes defined in (\ref{eq14}), and recall that the $\alpha$-stable fragmentation process, $\mathbf{F}^{(\alpha)} = (\mathbf{F}^{(\alpha)}(t), t \geq 0)$, is given by $\mathbf{F}^{(\alpha)}(t) = \mathbf{F}( - I_{\alpha}^{(t)})$, for $t \geq 0$. 

We prove that the processes $W_{n} = (W_{n}^{(t)}, t \geq 0)$ and $Y_{\alpha} = (Y_{\alpha}^{(t)}, t \geq 0)$ satisfy the conditions of Lemma \ref{lemma5}. Note that for all $t \geq 0$, $W_{n}^{(t)}(0) = Y_{\alpha}^{(t)}(0) = 0$. Moreover, Theorem \ref{Theo2} implies that $(W_{n}^{(t)}, t \geq 0) \rightarrow (Y_{\alpha}^{(t)}, t \geq 0)$, in distribution, as $n \rightarrow \infty$, in the space $\mathbb{D}(\mathbb{R}_{+}, \mathbb{D}([0,1], \mathbb{R}))$. We now verify that the process $Y_{\alpha}$ satisfies conditions (i), (ii) and (iii) of Lemma \ref{lemma5}. Indeed, (i) has been proven in Theorem \ref{Theo2}. The process $X_{\alpha}^{\rm br}$ has exchangeable increments due to the stationary and independent increments of the stable L\'evy process $X_{\alpha}$; see e.g., \cite[Chapters 11 and 16]{Kall2005}. Then, (ii) follows along the lines of the proof of Lemma 7 (i) in \cite{Bertoin2001} thanks to the property in 
(\ref{eq17}). To prove that $Y_{\alpha}^{(t)}$ fulfills condition (iii) for every $t \geq 0$, recall that the support of the Stieltjes measure ${\rm d}(- I_{\alpha}^{(t)})$ coincides with the ladder time set $\mathscr{L}^{\alpha}(t)$ of $Y_{\alpha}^{(t)}$, which is a random closed set with zero Lebesgue measure. The latter follows from \cite[Corollary 5, Chapter VII]{Bertoin1996} but alternatively, it can be deduced from (\ref{eq17}) by following the same argument as in \cite[Proof of Lemma 7]{Bertoin2001}. Since $\mathbf{F}(-I_{\alpha}^{(t)})$ is defined as the ranked sequence of the lengths of the open intervals in the canonical decomposition of $[0,1]/ \mathscr{L}^{\alpha}(t)$, condition (iii) follows.

Therefore, Lemma \ref{lemma5} shows the convergence of the finite-dimensional distributions of the sequence of processes $(\mathbf{F}_{n}^{(\alpha)})_{n \geq 1}$ to those of the process $\mathbf{F}^{(\alpha)}$. On the other hand, \cite[Corollary 6.1]{GbSvCe2025} (see also \cite[Remark 6.2]{GbSvCe2025} shows that $(\mathbf{F}_{n}^{(\alpha)})_{n \geq 1}$ is tight in $\mathbb{D}(\mathbb{R}_{+}, \mathbb{S})$, which concludes our proof.
\end{proof}

\section{Proof of Proposition \ref{Theo4}} \label{Sec6}

In this section, we prove Proposition \ref{Theo4}. The proof follows along the lines of the proof of Proposition 13 in Aldous and Pitman \cite{AldousPitmanI2000} (see also Theorem 3 in \cite{AldousPitman1998}). We provide enough details to convince the reader that everything can be carried out as in \cite{AldousPitmanI2000}, but also to make this work self contained. \\

\noindent \textbf{The $\alpha$-stable L\'evy tree.} Recall that an $\alpha$-stable L\'evy tree $\mathcal{T}_{\alpha} = (\mathcal{T}_{\alpha}, d_{\alpha}, \rho_{\alpha}, \mu_{\alpha})$ of index $\alpha \in (1,2]$ is a random compact rooted measure that arises naturally as the scaling limit of large $\alpha$-stable ${\rm GW}$-trees. More precisely, let $\mathbf{t}_{n}$ be an $\alpha$-stable ${\rm GW}$-tree, view it as a rooted metric measure tree $\mathbf{t}_{n} = (\mathbf{t}_{n}, d_{n}^{\text{gr}}, \rho_{n}, \mu_{n}^{\rm nod})$, where $\mathbf{t}_{n}$ is identified as its set of $n$ vertices $\{v_{1}, \dots, v_{n}\}$, $d_{n}^{\text{gr}}$ is the graph-distance on $\mathbf{t}_{n}$, $\rho_{n} \in \mathbf{t}_{n}$ is the root (the initial individual in the population) and $\mu_{n}^{\rm nod}  \coloneqq \frac{1}{n} \sum_{i=1}^{n} \delta_{v_{i}}$ is the uniform measure on the set of vertices  of $\mathbf{t}_{n}$; here $\delta_{v}$ is the Dirac measure in the point $v \in \mathbf{t}_{n}$. Let $(B_{n})_{n \geq 1}$ be a sequence of positive real numbers satisfying (\ref{eq10}) and consider the rescaled $\alpha$-stable ${\rm GW}$-tree $(B_{n}/n) \cdot \mathbf{t}_{n} = (\mathbf{t}_{n}, (B_{n}/n) \cdot d_{n}^{\text{gr}}, \rho_{n}, \mu_{n}^{\rm nod})$. Then it is well-known, by results of Aldous \cite{Al1993} and Duquesne \cite{Du2003}, that
\begin{eqnarray} \label{eq18}
(\mathbf{t}_{n}, (B_{n}/n) \cdot d_{n}^{\text{gr}}, \rho_{n}, \mu_{n}^{\rm nod}) \xrightarrow[ ]{d} (\mathcal{T}_{\alpha}, d_{\alpha}, \rho_{\alpha}, \mu_{\alpha}) , \hspace*{3mm} n \rightarrow \infty, 
\end{eqnarray}

\noindent for the {\sl pointed Gromov-Hausdorff-Prohorov} (pGHP) topology. (see for example  \cite[Proposition 9]{Miermont2009}, \cite[Theorem 2.5]{Abraham2013} and reference therein for background on the pGHP topology.) We list some useful properties of the $\alpha$-stable L\'evy tree and the rescaled $\alpha$-stable ${\rm GW}$-tree. 
\begin{itemize}
\item[({\bf T1})] The mass measure $\mu_{\alpha}$ is non-atomic and it is supported on ${\rm Lf}(\mathcal{T}_{\alpha})$, a.s.; see \cite[Theorem 4.6]{DuLegal2005}. 
\end{itemize}

\noindent For $k \in \mathbb{N}$, let $V_{1}^{n}, \dots, V_{k}^{n}$ be independent random vertices of $\mathbf{t}_{n}$ with common distribution $\mu_{n}^{\rm nod}$. Let $\mathcal{R}(\mathbf{t}_{n}, \mathbf{V}_{k}^{n})$ be the reduced subtree of $\mathbf{t}_{n}$ by its root and the vertices $\mathbf{V}^{n}_{k} = (V_{1}^{n}, \dots, V_{k}^{n})$ (i.e., $\mathcal{R}(\mathbf{t}_{n}, \mathbf{V}_{k}^{n})$ is viewed as the compact rooted metric space $(\mathcal{R}(\mathbf{t}_{n}, \mathbf{V}_{k}^{n}), d_{n}^{\text{gr}}, \rho)$, where $\mathcal{R}(\mathbf{t}_{n}, \mathbf{V}_{k}^{n})$ is identified as its set of vertices and the distance $d_{n}^{\text{gr}}$ is tacitly understood to be restricted to the appropriate space). Let also $(B_{n}/n) \cdot  \mathcal{R}(\mathbf{t}_{n}, \mathbf{V}_{k}^{n})$ be the space $\mathcal{R}(\mathbf{t}_{n}, \mathbf{V}_{k}^{n})$ with distances multiplied by $B_{n}/n$. Similarly, we let  $V_{1}, \dots, V_{k}$ be independent random points (leaves) of $\mathcal{T}_{\alpha}$ with common distribution $\mu_{\alpha}$, and write $\mathcal{R}(\mathcal{T}_{\alpha}, \mathbf{V}_{k})$ for the reduced subtree of $\mathcal{T}_{\alpha}$ by its root and the vertices $\mathbf{V}_{k} = (V_{1}, \dots, V_{k})$.  
\begin{itemize}
\item[({\bf T2})]  For every fixed $k \in \mathbb{N}$, $(B_{n}/n) \cdot  \mathcal{R}(\mathbf{t}_{n}, \mathbf{V}_{k}^{n}) \xrightarrow[ ]{d}  \mathcal{R}(\mathcal{T}_{\alpha}, \mathbf{V}_{k})$, as $n \rightarrow \infty$, for the {\sl pointed Gromov-Hausdorff} topology. This follows from (\ref{eq18}), \cite[Proposition 10]{Miermont2009} and \cite[Lemma 35]{Haas2012}.
\end{itemize}

\noindent Define the empirical (random) measures
\begin{eqnarray} \label{eq19}
\mu_{n,k}^{\rm nod}  \coloneqq \frac{1}{k} \sum_{i=1}^{k} \delta_{V_{i}^{n}} \hspace*{4mm} \text{and} \hspace*{4mm} \mu_{\alpha, k}  \coloneqq \frac{1}{k} \sum_{i=1}^{k} \delta_{V_{i}},
\end{eqnarray}

\begin{itemize}
\item[({\bf T3})]  The Glivenko-Cantelli Theorem implies that $\mu_{n,k}^{\rm nod} \rightarrow \mu_{n}^{\rm nod}$ and $\mu_{\alpha, k} \rightarrow \mu_{\alpha}$, almost surely, as $k \rightarrow \infty$, in the weakly sense.

\item[({\bf T4})]  Theorem 3 in \cite{Al1993} shows that the family of reduced subtrees $(\mathcal{R}(\mathcal{T}_{\alpha}, \mathbf{V}_{k}), k \in \mathbb{N})$ satisfies the so-called leaf-tight property, i.e.\ $\inf_{2 \leq i < \infty} d_{\alpha}( V_{1}, V_{i}) = 0$, almost surely.
\end{itemize}

\noindent \textbf{Exchangeable random partitions.} A partition $\Pi$ of $\mathbb{N}= \{1, 2, \dots \}$ is a countable collection $\Pi = (\Pi(i), i \in \mathbb{N})$ of pairwise disjoint subsets of $\mathbb{N}$ (also called blocks) such that $\bigcup_{i \in \mathbb{N}} \Pi(i) = \mathbb{N}$, which are always enumerated in increasing order of their least element, that is, $\min \Pi(i) \leq \min \Pi(j)$ for every $1 \leq i \leq j$, with the convention $\inf \emptyset = \infty$. For e.g., an equivalence relation $\sim$ on the set $\mathbb{N}$ can be identified with a partition of $\mathbb{N}$ into equivalence classes. In particular, a random equivalence relation on $\mathbb{N}$ can be identified with a random partition of $\mathbb{N}$. Let $\mathcal{P}_{\infty}$ be the set of partitions of the set of positive integers $\mathbb{N}$. Lemma 2.6 in \cite{Bertoin2006} shows that $\mathcal{P}_{\infty}$ can be endowed with an ultra-metric $d_{\mathcal{P}_{\infty}}$ such that $(\mathcal{P}_{\infty}, d_{\mathcal{P}_{\infty}})$ is compact. An exchangeable random partition $\Pi$ is a $\mathcal{P}_{\infty}$-valued random variable whose restriction $\Pi_{k} = \Pi |_{[k]}$ to the set $[k] \coloneqq \{1, \dots, k\}$  has an invariant distribution under the action of permutations of $[k]$, for every $k \in \mathbb{N}$. 

Following Kingman's theory \cite{Kingman1978}, we recall some useful properties of exchangeable random partitions. For $k \in \mathbb{N}$ and a partition $\Pi \in \mathcal{P}_{\infty}$, let $\Pi_{k} = (\Pi_{k}(i), i \in \mathbb{N})$ be the restriction of $\Pi$ to $[k]$, and let $\# \Pi_{k}^{\downarrow} = (\# \Pi_{k}^{\downarrow}(i), i \in \mathbb{N})$ be the decreasing rearrangement of the block sizes (number of elements) of $\Pi_{k}$ such that $\# \Pi_{k}^{\downarrow}(i) = 0$ whenever $\Pi_{k}$ has fewer than $i$ blocks. Let $\mathbb{S}_{\leq 1} \subset \mathbb{S}$ be the space of the elements of $\mathbb{S}$ with sum less than or equal to $1$. Recall also that $\mathbb{S}_{1} \subset \mathbb{S}$ denotes the space of the elements of $\mathbb{S}$ with sum $1$.

\begin{itemize}
\item[({\bf P1})] Let $\Pi$ be an exchangeable random partition. Theorem 2.1 in \cite{Bertoin2006} and the Fatou's lemma show that the asymptotic ranked frequencies (in decreasing order)
\begin{eqnarray*}
|\Pi(i)|^{\downarrow} \coloneqq \lim_{k \rightarrow \infty} \frac{\# \Pi_{k}^{\downarrow}(i)}{k}, \hspace*{3mm} \text{for} \hspace*{2mm} i \in \mathbb{N}, \hspace*{2mm}  \text{exist a.s. and} \hspace*{2mm} (|\Pi(i)|^{\downarrow}, i \in \mathbb{N}) \in \mathbb{S}_{\leq 1}.
\end{eqnarray*}

\item[({\bf P2})] $(|\Pi(i)|^{\downarrow}, i \in \mathbb{N}) \in \mathbb{S}_{1}$ a.s.\ if and only if $\{1\}$ is not a class (i.e., the singleton $\{1\}$ is not a block) of $\Pi$ a.s.; see \cite[Proposition 2.8]{Bertoin2006}.

\item[({\bf P3})] For each $n \in \mathbb{N} \cup \{\infty\}$, let $\Pi^{(n)}$ be an exchangeable random partition, and write $(|\Pi^{(n)}(i)|^{\downarrow}, i \in \mathbb{N})$ for the sequence of asymptotic ranked frequencies of its blocks in decreasing order. For $k \in \mathbb{N}$, let $\mathcal{P}_{k}$ be the set of partitions of $[k]$ endowed with the discrete topology. Then, Theorem 11 in \cite{AldousPitmanI2000} (see also, \cite[Proposition 2.9]{Bertoin2006}) implies that 
\begin{eqnarray*}
\Pi^{(n)} |_{[k]} \xrightarrow[ ]{d}  \Pi^{(\infty)} |_{[k]}, \hspace*{3mm} \text{as} \hspace*{2mm} n \rightarrow \infty, \hspace*{2mm} \text{for each} \hspace*{2mm} k \in \mathbb{N}, \hspace*{2mm} \text{in the space} \hspace*{2mm} \mathcal{P}_{k}
\end{eqnarray*}

\noindent if and only if
\begin{eqnarray} \label{eq20}
(|\Pi^{(n)}(i)|^{\downarrow}, i \in \mathbb{N}) \xrightarrow[ ]{d}  (|\Pi^{(\infty)}(i)|^{\downarrow}, i \in \mathbb{N}), \hspace*{3mm} \text{as} \hspace*{2mm} n \rightarrow \infty, \hspace*{2mm} \text{in the space} \hspace*{2mm} \mathbb{S}_{\leq 1},
\end{eqnarray}

\noindent where $\mathbb{S}_{\leq 1}$ is given the topology of pointwise convergence (or equivalently, the uniform distance in \cite[Proposition 2.1]{Bertoin2006} which makes $\mathbb{S}_{\leq 1}$ compact).
\end{itemize}

\begin{lemma} \label{corollary4}
Suppose that (\ref{eq20}) holds and that $(|\Pi^{(\infty)}(i)|^{\downarrow}, i \in \mathbb{N}) \in \mathbb{S}_{1}$ almost surely. Then, 
\begin{eqnarray*} 
(|\Pi^{(n)}(i)|^{\downarrow}, i \in \mathbb{N}) \xrightarrow[ ]{d}  (|\Pi^{(\infty)}(i)|^{\downarrow}, i \in \mathbb{N}), \hspace*{3mm} \text{as} \hspace*{2mm} n \rightarrow \infty, \hspace*{2mm} \text{in the space} \hspace*{2mm} (\mathbb{S}_{1}, \ell_{1}).
\end{eqnarray*}
\end{lemma}

\begin{proof}
The proof follows by a simple application of Fatou's lemma and Scheff\'e's lemma.
\end{proof}
 
\noindent \textbf{Fragmentation processes.} Following ideas of Aldous and Pitman \cite{AldousPitmanI2000}, the framework of exchangeable random partitions provides a different interpretation for the fragmentation processes associated to $\alpha$-stable L\'evy trees and $\alpha$-stable ${\rm GW}$-trees. 

Consider an $\alpha$-stable L\'evy tree $\mathcal{T}_{\alpha} = (\mathcal{T}_{\alpha}, d_{\alpha}, \rho_{\alpha}, \mu_{\alpha})$ together with a Poisson point process of cuts on its skeleton with intensity ${\rm d}t \otimes \lambda_{\alpha}({\rm d}v)$ on $[0, \infty) \times \mathcal{T}_{\alpha}$, where $\lambda_{\alpha}$ is the length measure associated to $\mathcal{T}_{\alpha}$. Recall that for all $t \geq 0$ we defined an equivalence relation $\sim_{t}$ on $\mathcal{T}_{\alpha}$ by saying that $v \sim_{t} w$, for $v, w \in \mathcal{T}_{\alpha}$, if and only if no atom of the Poisson process that has appeared before time $t$ belongs to the path $[v, w]$. We use the above to define a random equivalence relation on $\mathbb{N}$. Let $V_{1}, V_{2}, \dots$ be a sequence of independent random points of $\mathcal{T}_{\alpha}$ with common distribution $\mu_{\alpha}$. For $t \geq 0$ and $i,j \in \mathbb{N}$, we say $i \sim_{\alpha, t} j $ if and only if $V_{i} \sim_{t} V_{j}$. In particular, we let $\Pi_{\alpha}^{(t)} = (\Pi_{\alpha}^{(t)}(i), i \in \mathbb{N})$ be the random partition of $\mathbb{N}$ induced by the equivalence classes of the equivalence relation $\sim_{\alpha, t}$ on $\mathbb{N}$. 

\begin{lemma} \label{lemma6}
For every $t \geq 0$, $\Pi_{\alpha}^{(t)}$ is exchangeable. In particular, $\Pi_{\alpha}^{(t)}$ is proper a.s., i.e., the asymptotic ranked frequencies $(|\Pi_{\alpha}^{(t)}(i)|^{\downarrow}, i \in \mathbb{N})$ (in decreasing order) of $\Pi_{\alpha}^{(t)}$ belongs to $\mathbb{S}_{1}$ almost surely.
\end{lemma}

\begin{proof}
The first claim follows from the fact that for every $k \in \mathbb{N}$ the distribution of the reduced subtree $\mathcal{R}(\mathcal{T}_{\alpha}, \mathbf{V}_{k})$ of $\mathcal{T}_{\alpha}$ is invariant under any permutation of the points (leaves) $V_{1}, \dots, V_{k}$, i.e.\ $\mathcal{R}(\mathcal{T}_{\alpha}, \mathbf{V}_{k})$. To prove the second part, note that the probability that $1 \sim_{\alpha, t} j$ is $\exp(-td_{\alpha}(V_{1}, V_{j}))$, for $j \geq 2$. Then ({\bf T4}) implies that $\{1\}$ is not a class a.s., and our claim follows from ({\bf P2}). 
\end{proof}

\begin{corollary} \label{corollary2}
For every $t \geq 0$, we have that $\mathbf{F}_{\mathcal{T}_{\alpha}}(t) = (|\Pi_{\alpha}^{(t)}(1)|^{\downarrow}, |\Pi_{\alpha}^{(t)}(2)|^{\downarrow}, \dots )$ almost surely. 
\end{corollary}

\begin{proof}
For $k \in \mathbb{N}$, let $\Pi_{\alpha, k}^{(t)} = (\Pi_{\alpha, k}^{(t)}(i), i \in \mathbb{N})$ be the restriction of $\Pi_{\alpha}^{(t)}$ to $[k]$, and let $\# \Pi_{\alpha, k}^{(t), \downarrow} = (\#\Pi_{\alpha, k}^{(t), \downarrow}(i), i \in \mathbb{N})$ be the decreasing rearrangement of the block sizes of $\Pi_{\alpha, k}^{(t)}$ such that $\#\Pi_{\alpha, k}^{(t), \downarrow}(i) = 0$ whenever $\Pi_{\alpha, k}^{(t)}$ has fewer than $i$ blocks. Let $A_{\alpha, 1}^{(t)}, A_{\alpha, 2}^{(t)}, \dots$ be the distinct equivalence classes for $\sim_{t}$. Then, $(\#\Pi_{\alpha, k}^{(t), \downarrow}(1), \#\Pi_{\alpha, k}^{(t), \downarrow}(2), \dots)$ is equal to the ranked vector $(\mu_{\alpha,k}(A_{\alpha, 1}^{(t)}), \mu_{\alpha,k}(A_{\alpha, 2}^{(t)}), \dots)$ in decreasing order. Thus, our claim follows from ({\bf T3}) and ({\bf P1}). 
\end{proof}

Consider now the (rescaled) $\alpha$-stable ${\rm GW}$-tree $(B_{n}/n) \cdot \mathbf{t}_{n} = (\mathbf{t}_{n}, (B_{n}/n) \cdot d_{n}^{\text{gr}}, \rho_{n}, \mu_{n}^{\rm nod})$, where $(B_{n})_{n \geq 1}$ is a sequence of positive real numbers satisfying (\ref{eq10}). For $t \geq 0$, recall that the fragmentation forest at time $s_{n}(t) = 1 - (B_{n}/n)t$, that is $\mathbf{f}_{n}(s_{n}(t))$, is obtained by keeping those edges in $\mathbf{t}_{n}$ with uniform weight smaller than $s_{n}(t)$. As for the fragmentation process of the $\alpha$-stable L\'evy tree, we can define a random equivalence relation on $\mathbb{N}$. Let $V_{1}^{n}, V_{2}^{n}, \dots$ be a sequence of independent random vertices of $\mathbf{t}_{n}$ with common distribution $\mu_{n}^{\rm nod}$. For $t \geq 0$ and $i,j \in \mathbb{N}$, we say 
$i \sim_{n, t} j$ if and only if there is no cut edge on the path from $V_{i}^{n}$ to $V_{j}^{n}$ before time $s_{n}(t)$. In particular, we let $\Pi_{n}^{(t)} = (\Pi_{n}^{(t)}(i), i \in \mathbb{N})$ be the random partition of $\mathbb{N}$ induced by the equivalence classes of the equivalence relation $\sim_{n, t}$ on $\mathbb{N}$. 

\begin{lemma} \label{lemma7}
For every $t \geq 0$, $\Pi_{n}^{(t)}$ is exchangeable. In particular, $\mathbf{F}_{n}^{(\alpha)}(t) = (|\Pi_{n}^{(t)}(1)|^{\downarrow}, |\Pi_{n}^{(t)}(2)|^{\downarrow}, \dots )$ a.s., where $(|\Pi_{n}^{(t)}(i)|^{\downarrow}, i \in \mathbb{N})$ are the asymptotic ranked frequencies of $\Pi_{n}^{(t)}$ in decreasing order. 
\end{lemma}

\begin{proof}
This follows along the lines of the proofs of Lemma \ref{lemma6} and Corollary \ref{corollary2} 
\end{proof}

Now we are able to prove Proposition \ref{Theo4}. 

\begin{proof}[Proof of Proposition \ref{Theo4}] 
Let $\mathbf{t}_{n}$ be an $\alpha$-stable ${\rm GW}$-tree, and for every fixed $t \geq 0$, view the (time-scaled) continuous cutting-down procedure of $\mathbf{t}_{n}$ as a (rescaled) Bernoulli process of cuts on its set of edges, that is, every edge of $\mathbf{t}_{n}$ is cut at time $t$ with probability $(B_{n}/n)t$. Then, at time $t \geq 0$, the sequence of sizes of the connected components of $\mathbf{t}_{n}$ in decreasing order and renormalized by a factor $1/n$ is given by $\mathbf{F}_{n}^{(\alpha)}(t)$. For every $k \in \mathbb{N}$ fixed, it should be clear that ({\bf T2}) implies that, as $n \rightarrow \infty$, the above (rescaled) Bernoulli process of cuts on $\mathbf{t}_{n}$ (viewed as a rooted metric measure tree) up to time $t$ and restricted to $\mathcal{R}(\mathbf{t}_{n}, \mathbf{V}_{k}^{n})$ converges (in distribution) to the Poisson point process of cuts on the skeleton of $\mathcal{T}_{\alpha}$ with intensity ${\rm d}s \otimes \lambda_{\alpha}({\rm d}v)$ restricted to $[0,t] \times \mathcal{R}(\mathcal{T}_{\alpha}, \mathbf{V}_{k})$. In fact, this convergence holds jointly with that in ({\bf T2}).  For every $t \geq 0$, it follows that
\begin{eqnarray*}
\Pi_{n}^{(t)} |_{[k]} \xrightarrow[ ]{d}  \Pi^{(t)}_{\alpha} |_{[k]}, \hspace*{3mm} \text{as} \hspace*{2mm} n \rightarrow \infty, \hspace*{2mm} \text{for each} \hspace*{2mm} k \in \mathbb{N}, \hspace*{2mm} \text{in the space} \hspace*{2mm} \mathcal{P}_{k}.
\end{eqnarray*}

\noindent Property $({\bf P3})$, Lemma \ref{lemma6}, Corollary \ref{corollary2} and Lemma \ref{lemma7} imply that
\begin{eqnarray*}
\mathbf{F}_{n}^{(\alpha)}(t) \xrightarrow[ ]{d}  \mathbf{F}_{\mathcal{T}_{\alpha}}(t), \hspace*{3mm} \text{as} \hspace*{2mm} n \rightarrow \infty, \hspace*{2mm} \text{in the space} \hspace*{2mm} \mathbb{S}_{\leq 1},
\end{eqnarray*}

\noindent where $\mathbb{S}_{\leq 1}$ is given the topology of pointwise convergence. Since Lemma \ref{lemma6} also shows that $\mathbf{F}_{\mathcal{T}_{\alpha}}(t) \in \mathbb{S}_{1}$ a.s., Lemma \ref{corollary4} entails that the above convergence holds in $(\mathbb{S}, \ell^{1})$. This shows the convergence of the one-dimensional distribution of $\mathbf{F}_{n}^{(\alpha)}$ to $\mathbf{F}_{\mathcal{T}_{\alpha}}$. In general, the same argument can be used to obtain the convergence of the finite-dimensional distributions thanks to the convergence of the (rescaled) Bernoulli process of cuts to the Poisson point process of cuts. 

Finally, Proposition \ref{Theo4} follows from Theorem \ref{Theo3}. 
\end{proof}

\paragraph{Acknowledgements.}
This work is supported by the Knut and Alice Wallenberg
Foundation, a grant from the Swedish Research Council and The Swedish Foundations' starting grant from Ragnar S\"oderbergs Foundation.


\begin{thebibliography}{10}

\bibitem{Abraham2013}
R.~Abraham, J.-F. Delmas, and P.~Hoscheit, \emph{A note on the
  {G}romov-{H}ausdorff-{P}rokhorov distance between (locally) compact metric
  measure spaces}, Electron. J. Probab. \textbf{18} (2013), no. 14, 21.
  \MR{3035742}

\bibitem{Abraham2002f}
R.~Abraham and L.~Serlet, \emph{Poisson snake and fragmentation}, Electron. J.
  Probab. \textbf{7} (2002), no. 17, 15. \MR{1943890}

\bibitem{Al1993}
D.~Aldous, \emph{The continuum random tree. {III}}, Ann. Probab. \textbf{21}
  (1993), no.~1, 248--289. \MR{1207226}

\bibitem{Aldous1997}
D.~Aldous, \emph{Brownian excursions, critical random graphs and the
  multiplicative coalescent}, Ann. Probab. \textbf{25} (1997), no.~2, 812--854.
  \MR{1434128}

\bibitem{AldousPitman1998}
D.~Aldous and J.~Pitman, \emph{The standard additive coalescent}, Ann. Probab.
  \textbf{26} (1998), no.~4, 1703--1726. \MR{1675063}

\bibitem{AldousPitmanI2000}
D.~Aldous and J.~Pitman, \emph{Inhomogeneous continuum random trees and the entrance boundary
  of the additive coalescent}, Probab. Theory Related Fields \textbf{118}
  (2000), no.~4, 455--482. \MR{1808372}

\bibitem{Bertoin1996}
J.~Bertoin, \emph{L\'{e}vy processes}, Cambridge Tracts in Mathematics, vol.
  121, Cambridge University Press, Cambridge, 1996. \MR{1406564}

\bibitem{Bertoin2000}
J.~Bertoin, \emph{A fragmentation process connected to {B}rownian motion}, Probab.
  Theory Related Fields \textbf{117} (2000), no.~2, 289--301. \MR{1771665}

\bibitem{Bertoin2001}
J.~Bertoin, \emph{Eternal additive coalescents and certain bridges with
  exchangeable increments}, Ann. Probab. \textbf{29} (2001), no.~1, 344--360.
  \MR{1825153}

\bibitem{BertoinS2002}
J.~Bertoin, \emph{Self-similar fragmentations}, Ann. Inst. H. Poincar\'{e} Probab.
  Statist. \textbf{38} (2002), no.~3, 319--340. \MR{1899456}

\bibitem{Bertoin2006}
J.~Bertoin, \emph{Random fragmentation and coagulation processes}, Cambridge
  Studies in Advanced Mathematics, vol. 102, Cambridge University Press,
  Cambridge, 2006. \MR{2253162}

\bibitem{BertoinPer2013}
J.~Bertoin, \emph{Almost giant clusters for percolation on large trees with
  logarithmic heights}, J. Appl. Probab. \textbf{50} (2013), no.~3, 603--611.
  \MR{3102504}

\bibitem{BertoinMiermont2006C}
J.~Bertoin and G.~Miermont, \emph{Asymptotics in {K}nuth's parking problem for
  caravans}, Random Structures Algorithms \textbf{29} (2006), no.~1, 38--55.
  \MR{2238028}

\bibitem{GbSvCe2025}
G.~{Berzunza Ojeda}, C.~{Holmgren}, and S.~{Janson}, \emph{{A tightness
  criterion for fragmentations}}, (2025+), arXiv:2507.05102.

\bibitem{Billi1999}
P.~Billingsley, \emph{Convergence of probability measures}, second ed., Wiley
  Series in Probability and Statistics: Probability and Statistics, John Wiley
  \& Sons, Inc., New York, 1999, A Wiley-Interscience Publication. \MR{1700749}

\bibitem{Brou2016}
N.~Broutin and J.-F. Marckert, \emph{A new encoding of coalescent processes:
  applications to the additive and multiplicative cases}, Probab. Theory
  Related Fields \textbf{166} (2016), no.~1-2, 515--552. \MR{3547745}

\bibitem{Chasing2002}
P.~Chassaing and G.~Louchard, \emph{Phase transition for parking blocks,
  {B}rownian excursion and coalescence}, Random Structures Algorithms
  \textbf{21} (2002), no.~1, 76--119. \MR{1913079}

\bibitem{ChassainSv2001}
P.~Chassaing and S.~Janson, \emph{A {V}ervaat-like path transformation for the
  reflected {B}rownian bridge conditioned on its local time at 0}, Ann. Probab.
  \textbf{29} (2001), no.~4, 1755--1779. \MR{1880241}

\bibitem{Chaumont1997}
L.~Chaumont, \emph{Excursion normalis\'{e}e, m\'{e}andre et pont pour les
  processus de {L}\'{e}vy stables}, Bull. Sci. Math. \textbf{121} (1997),
  no.~5, 377--403. \MR{1465814}

\bibitem{Du2003}
T.~Duquesne, \emph{A limit theorem for the contour process of conditioned
  {G}alton-{W}atson trees}, Ann. Probab. \textbf{31} (2003), no.~2, 996--1027.
  \MR{1964956}

\bibitem{DuLegall2002}
T.~Duquesne and J.-F. Le~Gall, \emph{Random trees, {L}\'{e}vy processes and
  spatial branching processes}, Ast\'{e}risque (2002), no.~281, vi+147.
  \MR{1954248}

\bibitem{DuLegal2005}
T.~Duquesne and J.-F. Le~Gall, \emph{Probabilistic and fractal aspects of {L}\'{e}vy trees}, Probab.
  Theory Related Fields \textbf{131} (2005), no.~4, 553--603. \MR{2147221}

\bibitem{EvansPitman1998}
S.~N. Evans and J.~Pitman, \emph{Construction of {M}arkovian coalescents}, Ann.
  Inst. H. Poincar\'{e} Probab. Statist. \textbf{34} (1998), no.~3, 339--383.
  \MR{1625867}

\bibitem{Feller}
W.~Feller, \emph{An introduction to probability theory and its applications.
  {V}ol. {II}}, Second edition, John Wiley \& Sons, Inc., New
  York-London-Sydney, 1971. \MR{0270403}

\bibitem{Haas2012}
B.~Haas and G.~Miermont, \emph{Scaling limits of {M}arkov branching trees with
  applications to {G}alton-{W}atson and random unordered trees}, Ann. Probab.
  \textbf{40} (2012), no.~6, 2589--2666. \MR{3050512}

\bibitem{jacod2003}
J.~Jacod and A.~N. Shiryaev, \emph{Limit theorems for stochastic processes},
  second ed., Grundlehren der Mathematischen Wissenschaften [Fundamental
  Principles of Mathematical Sciences], vol. 288, Springer-Verlag, Berlin,
  2003. \MR{1943877}

\bibitem{Kall2005}
O.~Kallenberg, \emph{Foundations of modern probability}, second ed.,
  Probability and its Applications (New York), Springer-Verlag, New York, 2002.
  \MR{1876169}

\bibitem{Kingman1978}
J.~F.~C. Kingman, \emph{The representation of partition structures}, J. London
  Math. Soc. (2) \textbf{18} (1978), no.~2, 374--380. \MR{509954}

\bibitem{Leje2005}
J.-F. Le~Gall, \emph{Random trees and applications}, Probab. Surv. \textbf{2}
  (2005), 245--311. \MR{2203728}

\bibitem{lushnikov1978}
A.~A. Lushnikov, \emph{Some new aspects of coagulation theory}, Izv. Akad. Nauk
  SSSR, Ser. Fiz. Atmosfer. I Okeana \textbf{14} (1978), no.~10, 738--743.

\bibitem{MarckertWang2019}
J.-F. Marckert and M.~Wang, \emph{A new combinatorial representation of the
  additive coalescent}, Random Structures Algorithms \textbf{54} (2019), no.~2,
  340--370. \MR{3912100}

\bibitem{Marcus1968}
A.~H. Marcus, \emph{Stochastic coalescence}, Technometrics \textbf{10} (1968),
  133--143. \MR{223151}

\bibitem{Miermont2001}
G.~Miermont, \emph{Ordered additive coalescent and fragmentations associated to
  {L}evy processes with no positive jumps}, Electron. J. Probab. \textbf{6}
  (2001), no. 14, 33. \MR{1844511}

\bibitem{MiermontII}
G.~Miermont, \emph{Self-similar fragmentations derived from the stable tree. {II}.
  {S}plitting at nodes}, Probab. Theory Related Fields \textbf{131} (2005),
  no.~3, 341--375. \MR{2123249}

\bibitem{Miermont2009}
G.~Miermont, \emph{Tessellations of random maps of arbitrary genus}, Ann. Sci.
  \'{E}c. Norm. Sup\'{e}r. (4) \textbf{42} (2009), no.~5, 725--781.
  \MR{2571957}

\bibitem{Ne1986}
J.~Neveu, \emph{Arbres et processus de {G}alton-{W}atson}, Ann. Inst. H.
  Poincar\'{e} Probab. Statist. \textbf{22} (1986), no.~2, 199--207.
  \MR{850756}

\bibitem{Otter1949}
R.~Otter, \emph{The multiplicative process}, Ann. Math. Statistics \textbf{20}
  (1949), 206--224. \MR{30716}

\bibitem{Perman1992}
M.~Perman, J.~Pitman, and M.~Yor, \emph{Size-biased sampling of {P}oisson point
  processes and excursions}, Probab. Theory Related Fields \textbf{92} (1992),
  no.~1, 21--39. \MR{1156448}

\bibitem{Pitman1999}
J.~Pitman, \emph{Coalescent random forests}, J. Combin. Theory Ser. A
  \textbf{85} (1999), no.~2, 165--193. \MR{1673928}

\bibitem{PitmanYor1997}
J.~Pitman and M.~Yor, \emph{The two-parameter {P}oisson-{D}irichlet
  distribution derived from a stable subordinator}, Ann. Probab. \textbf{25}
  (1997), no.~2, 855--900. \MR{1434129}

\bibitem{Prim}
R.~C. {Prim}, \emph{Shortest connection networks and some generalizations}, The
  Bell System Technical Journal \textbf{36} (1957), no.~6, 1389--1401.

\bibitem{Tak1967}
L.~Tak\'{a}cs, \emph{On combinatorial methods in the theory of stochastic
  processes}, Proc. {F}ifth {B}erkeley {S}ympos. {M}ath. {S}tatist. and
  {P}robability ({B}erkeley, {C}alif., 1965/66), Univ. California Press,
  Berkeley, Calif., 1967, pp.~Vol. II: Contributions to Probability Theory,
  Part 1, pp. 431--447. \MR{0214129}

\bibitem{thvenin2019geometric}
P.~Th\'{e}venin, \emph{A geometric representation of fragmentation processes on
  stable trees}, Ann. Probab. \textbf{49} (2021), no.~5, 2416--2476.
  \MR{4317709}

\bibitem{Veervat1979}
W.~Vervaat, \emph{A relation between {B}rownian bridge and {B}rownian
  excursion}, Ann. Probab. \textbf{7} (1979), no.~1, 143--149. \MR{515820}

\end{thebibliography}

\providecommand{\bysame}{\leavevmode\hbox to3em{\hrulefill}\thinspace}
\providecommand{\MR}{\relax\ifhmode\unskip\space\fi MR }
\providecommand{\MRhref}[2]{%
  \href{http://www.ams.org/mathscinet-getitem?mr=#1}{#2}
}
\providecommand{\href}[2]{#2}

\end{document}